\documentclass[reqno]{amsart}

\usepackage{general,resistance,ifthen,txfonts}
\usepackage{cancel,xy}

\usepackage[bookmarks, colorlinks=true, pdfstartview=FitV, linkcolor=blue, citecolor=blue, urlcolor=blue]{hyperref}
\usepackage{cite}  

\usepackage[left=1.4in,right=1.4in,top=1.2in,bottom=1.2in]{geometry}

\newcommand{\ipH}[2][{\sH}]{\left\la #2 \right\ra_{\scalebox{0.50}{\ensuremath{#1}}}}
\newcommand{\proj}[2][{\sH}]{\operatorname{proj}_{\scalebox{0.50}{\ensuremath{#1}}} \clspn{#2}}
\newcommand{\bH}{\ensuremath{\mathbb{H}}\xspace}
\DeclareMathOperator*{\punt}{\bullet} 
\newcommand{\vertex}[1]{\raisebox{-4mm}{$\punt\limits_{#1}$}}   

\numberwithin{equation}{section} \numberwithin{theorem}{section}

\begin{document}


\title{Spectral reciprocity and matrix representations of unbounded operators}

\author[P. Jorgensen]{Palle E. T. Jorgensen}
\address{University of Iowa, Iowa City, IA 52246-1419 USA}
\email{\href{mailto:jorgen@math.uiowa.edu}{jorgen@math.uiowa.edu}}
\urladdr{\href{http://www.math.uiowa.edu/~jorgen/}{http://www.math.uiowa.edu/$\sim$jorgen/}}

\author[E. Pearse]{Erin P. J. Pearse}
\address{University of Iowa, Iowa City, IA 52246-1419 USA}
\email{\href{mailto:erin-pearse@uiowa.edu}{erin-pearse@uiowa.edu}}
\urladdr{\href{http://www.math.uiowa.edu/~epearse/}{http://www.math.uiowa.edu/$\sim$epearse/}}

\thanks{The work of PETJ was partially supported by NSF grant DMS-0457581. The work of EPJP was partially supported by the University of Iowa Department of Mathematics NSF VIGRE grant DMS-0602242.}



\date{\bf\today.}

\begin{abstract}
  We study a family of unbounded Hermitian operators in Hilbert space which generalize the usual graph-theoretic discrete Laplacian. For an infinite discrete set $X$, we consider operators acting on Hilbert spaces of functions on $X$, and their representations as infinite matrices; the focus is on $\ell^2(X)$, and the energy space $\mathcal{H}_{\mathcal E}$. In particular, we prove that these operators are always essentially self-adjoint on $\ell^2(X)$, but may fail to be essentially self-adjoint on $\mathcal{H}_{\mathcal E}$. In the general case, we examine the von Neumann deficiency indices of these operators and explore their relevance in mathematical physics. Finally we study the spectra of the $\mathcal{H}_{\mathcal E}$ operators with the use of a new approximation scheme.
\end{abstract}




\maketitle

\setcounter{tocdepth}{1} {\footnotesize \tableofcontents}

\allowdisplaybreaks



\section{Introduction}
\label{sec:introduction}

This paper concerns the study of unbounded operators with dense domain in a Hilbert space, and their representation in terms of (infinite) matrices. In particular, Theorem~\ref{thm:matrix-conds-for-essential-selfadjointness} shows that a ``matrix Laplacian'' is always essentially self-adjoint on $\ell^2(X)$. This class of operators is a generalization of the usual discrete Laplacian from graph theory. We also show how the same operator is \emph{not} essentially self-adjoint with respect to the energy space, where the usual $\ell^2$ inner product is replaced by a alternative inner product given by a natural (quadratic) energy form. We give an axiomatic description of such energy spaces \bH and derive several properties of such spaces from this axiom system. We also prove a spectral reciprocity theorem (Theorem~\ref{thm:spectral-reciprocity}) which establishes an inverse relationship between the spectrum of the Laplacian (as an operator on \bH) and the spectrum of a matrix operator $M$ (as an operator on $\ell^2(X)$). The matrix entries of $M$ are defined in terms of a reproducing kernel for \bH. 


\pgap


The question of infinite matrix representations of geometric operators arose in a recent project \cite{DGG,ERM,bdG,RBIN,OTERN, Multipliers, Interpolation, LPS}, where the authors study resistance networks and their applications. In these papers, the authors found that that crucial properties of resistance networks may be understood with the use of an associated Laplace operator \Lap, and its various representations. The harmonic analysis of resistance forms in the self-adjoint case is worked out in great detail in \cite{Kig01,Kig03,Kig09} via an elegant potential-theoretic approach. As noted in \cite{DGG,Kig03}, while tempting to study \Lap as an operator in $\ell^2$, this approach turns out to miss much of the harmonic analysis for the given resistance network. The emphasis of the present paper is on the situation where \Lap may be only essentially self-adjoint, or possibly even have different self-adjoint extensions. While $\ell^2$ detects important spectral data of the Laplacian (and thus also some related combinatorial properties), it sees strikingly little of the geometry of the resistance network, in comparison to the spectral theory of \Lap in the energy Hilbert space \HE; see \cite{bdG,RBIN, Multipliers}. 
For a particular problem, the choice of Hilbert space ($\ell^2$ or \HE or possibly even something else) will play a crucial role in allowing one to extract global properties of both the operator and the underlying space (network, graph, or more general set). 
While \cite{DGG,ERM,bdG,RBIN,OTERN, Multipliers, Interpolation, LPS} focus on the energy Hilbert space \HE, the present paper examines the $\ell^2$ theory in more depth. Among other things, we examine the deeper reason for why \Lap is essentially self-adjoint in $\ell^2$ but not in $\HE$, 
and that these two scenarios exhibit drastically different boundary conditions in the sense of von Neumann's deficiency indices \cite{vN32a, DuSc88}. 

\subsection{Outline}
\S\ref{sec:self-adjoint-operators-in-L2} discusses some issues related to the matrix representation of unbounded operators with dense domain in a Hilbert space. Special emphasis is placed on a class of operators which we call matrix Laplacians, as they generalize the usual discrete Laplace operator on graphs, and can be represented in terms of matrix multiplication by an infinite matrix $A$. Most of \S\ref{sec:Matrices-and-self-adjoint-operators-in-L2} is devoted to Theorem~\ref{thm:matrix-conds-for-essential-selfadjointness} (and the lemmas required for its proof), in which we prove that a matrix Laplacian $\Lap_A$ acting on $\ell^2(X)$ is always essentially self-adjoint.  

\S\ref{sec:repkernel-energy-space} gives an axiomatic presentation of a class of reproducing kernel Hilbert spaces (RKHS). The matrices considered in \S\ref{sec:Matrices-and-self-adjoint-operators-in-L2} give rise to an RKHS of this type, and the RKHS studied in \S\ref{sec:Matrices-and-self-adjoint-operators-in-the-energy-space} is a special case of this class.

\S\ref{sec:Matrices-and-self-adjoint-operators-in-the-energy-space} considers a special case of the reproducing kernel Hilbert spaces of \S\ref{sec:repkernel-energy-space} which the authors have previously studied in \cite{DGG,ERM,bdG,RBIN,OTERN}, namely, the energy space associated to a resistance network. \S\ref{sec:Lap-fails-to-be-essentially-self-adjoint-on-HE} describes how \Lap can fail to be essentially self-adjoint (as an operator on \sH) by explicitly computing an example with deficiency indices $(1,1)$ and giving a formula for the defect vector (which is also shown to be bounded and of finite energy).

\S\ref{sec:finite-approximants} returns to the consideration of the matrix $M$ with entries \smash{$M_{xy} = \ipH{v_x,v_y}$}, which first appeared as a positive semidefinite function on $X \times X$ in \S\ref{sec:repkernel-energy-space}. The key result in this section is Theorem~\ref{thm:spectral-reciprocity}, which establishes a form of spectral reciprocity between \Lap and $D$, the diagonalization of $M$. The exact relationship between $M$ and \Lap is made precise in Corollary~\ref{thm:Green-is-inverse-of-Lap-strike-o}; see also Remark~\ref{rem:condensation-of-spec-meas}.

\subsection{Basic definitions and facts for unbounded operators on a Hilbert space}
\label{sec:basics-for-unbounded-operators-on-a-Hilbert-space}

In this section, we recall some terms and basic results from the theory of unbounded operators on a Hilbert space. This material can be found in a standard reference, such as \cite{DuSc88} or \cite{ReedSimonI}.

Consider an operator $T$ acting on a complex Hilbert space \bH. 
We will use $\sD =\dom T$ to denote the domain of the operator $T$, so \sD is always a dense linear subspace of \bH. 

\begin{defn}\label{def:Hermitian}
  The operator $T$ is \emph{Hermitian} or (\emph{symmetric} or \emph{formally self-adjoint}) iff
\linenopax
\begin{align*}
  \la Tu, v\ra = \la u, Tv\ra, 
  \qq \text{for all } u, v \in \sD.
\end{align*}
\end{defn}

\begin{defn}\label{def:adjoint}
  Let $T$ be a densely defined operator in a Hilbert space \bH. Define
  \linenopax
  \begin{align*}
    \dom(T^\ad)
    := \{v \in \bH \suth \exists C<\iy \text{ s.t. } 
         |\la v, Tu\ra|\leq C\|u\|, \,\forall u \in \dom(T)\}.
  \end{align*}
  In that case, by Riesz's Theorem, there exists a unique $w \in \bH$ such that 
  \linenopax
  \begin{align*}
    \la v,Tu\ra = \la w,u\ra, \,\forall u \in \dom(T),
  \end{align*}
   and we set $T^\ad v = w$. Then $T^\ad$ is the \emph{adjoint} of $T$.
\end{defn}

\begin{defn}\label{def:operator-extension}
  If $S$ and $T$ are operators with dense domains $\dom S \ci \dom T \ci \bH$, then $T$ is an \emph{extension} of $S$ iff $T$ restricted to $\dom S$ coincides with $S$. This is typically denoted $S \ci T$, where the inclusion refers to the containment of the respective operator graphs.
\end{defn}

\begin{defn}\label{def:essentially-selfadjoint}
  An operator $T$ with dense domain $\sD \ci \bH$ is said to be \emph{self-adjoint} iff $T^\ad = T$. The operator $T$ is said to be \emph{essentially self-adjoint} iff it has a unique self-adjoint extension.%
    \footnote{In which case, that unique self-adjoint extension is just the closure of $T$, in accordance with Lemma~\ref{thm:essential-self-adjointness}(i).}
\end{defn}

\begin{defn}\label{def:semibounded}
  If $T$ is a densely defined operator on \bH, then $T$ is \emph{semibounded} iff 
  \linenopax
  \begin{align}\label{eqn:def:positive-semibounded}
    \la u, Tu\ra \geq 0, \q \text{for all } u \in \dom(T),  
  \end{align}
  or if the reverse inequality is true.  If \eqref{eqn:def:positive-semibounded} holds, we say that $T$ is a \emph{positive semidefinite operator}.
\end{defn}

\begin{lemma}\label{thm:Hermitian=extension}
  If $T$ is an operator on a Hilbert space, then $T$ is Hermitian iff $T \ci T^\ad$.
\end{lemma}

\begin{lemma}
  \label{thm:essential-self-adjointness}
  \label{thm:essentially-sa-iff-no-defect}
  Let $T$ be a Hermitian operator on a Hilbert space. Then the essential self-adjointness of $T$ is equivalent to 
  \begin{enumerate}[(i)]
    \item the closure of $T$ is self-adjoint, or 
    \item $\ker(T^\ad \pm \ii) = \{0\}$.
  \end{enumerate}
  If $T$ is Hermitian and semibounded, then $T$ is essentially self-adjoint iff 
  \begin{enumerate}[(i)]
    \item[(iii)] $\ker(\id + T^\ad) = \{0\}$, or equivalently, the range $\ran(\id + A)$ is dense in \bH.
  \end{enumerate}
\end{lemma}

Since the Laplace operator $T = \Lap$ discussed below is semibounded, we find it most convenient to use criterion (iii). In this case, $T$ is essentially self-adjoint if and only if the eigenvalue problem 
   $ T^\ad v = -v$
has only the trivial solution $v=0$.

Since the property of semiboundedness is critical in the following, it is shown for the operator \Lap acting on the Hilbert space $\bH = \ell^2(X)$ in Lemma~\ref{thm:<u,Au>=E(u)}. The semiboundedness of the operator \Lap acting on the reproducing kernel energy Hilbert spaces $\bH=\sH$ of \S\ref{sec:repkernel-energy-space} is shown in Lemma~\ref{thm:semibounded}.


\section{Unbounded operators on the separable Hilbert space $\bH=\ell^2(X)$}
\label{sec:self-adjoint-operators-in-L2}

We stress the interplay between operators defined on a dense domain in Hilbert space \bH on the one hand and their matrix representation on the other. The questions we address arise only in the case when \bH is infinite dimensional, so we will be considering infinite matrices. Once \bH is given, we may select an orthonormal basis $B$. Selecting an index set $X$ for $B$, we note that \bH is then isometrically isomorphic to $\ell^2(X)$ = the square summable sequences indexed by $X$. We will restrict to the case when $X$ is countable, i.e., \bH separable. Our infinite matrices will then have rows and columns indexed by the set $X$.

In some of our applications, the set $X$ will be the set of vertices on some weighted graph $(\Graph,\cond)$ with \cond some (positive and symmetric) function defined on the set of edges in \Graph. In this case, $\ell^2(X)$ will not capture the important data for $(\Graph,\cond)$ and we use a second Hilbert space \HE defined from an energy form for $(\Graph,\cond)$. In this case, there is a natural Laplace operator \Lap associated with $(\Graph,\cond)$. It turns out that it will have quite different properties depending on whether it is computed in $\ell^2(X)$ or in \HE. The matrix representations for \Lap will be different for the two Hilbert spaces. Understanding the interrelations of these two versions of \Lap in terms of their matrix representations is a main theme of this paper. 

\subsection{Matrix representations of operators on $\ell^2(X)$}
\label{sec:ell2(X)}

This paper is primarily concerned with the case when $X$ is a countably infinite set, in which case $\bH = \ell^2(X)$ is separable. Here, $\ell^2(X)=\ell^2(X,\gm)$ where \gm is counting measure, and we use the usual inner product
\linenopax
  \begin{align*}
  \la u, v\ra_{\ell^2} := \sum_{x \in X} \cj{u(x)} v(x),
\end{align*}
and let $T: \sD \to \bH$ be a linear operator on \bH. 
For the Hilbert space $\ell^2(X)$, we use the orthonormal basis (onb) of Dirac masses $\{\gd_x\}_{x \in X}$ given by
\linenopax
\begin{align*}
  \gd_x(y) = 
  \begin{cases} 
    1, &\text{if }y=x \\
    0, &\text{if }y \neq x.
  \end{cases}
\end{align*}
A function $u$ on $X$ will be viewed as a column vector. If $A = (a_{x,y})_{x,y \in X}$ is an \bR-valued function on $X \times X$, then $T_A(u) = Au$ is defined by
\linenopax
\begin{align}\label{eqn:matrix-mult}
  T_A:u \mapsto Au,
  \q\text{where}\q
  (Au)(x) = \sum_{y \in X} a_{x,y} u(y)
\end{align}
(i.e., by matrix multiplication) with the understanding that the summation in the right-hand side of \eqref{eqn:matrix-mult} is absolutely convergent. Henceforth, we describe an object such as $A$ as an \emph{infinite matrix} with rows and columns indexed by $X$.

\begin{defn}\label{def:c_0}
  The collection of all finitely-supported functions on $X$ is 
  \linenopax
  \begin{align}\label{eqn:def:c_0}
    c_0(X) := \{u:X \to \bC \suth u\evald{X\less F} = 0 \text{ for some finite subset } F \ci X\}. 
  \end{align}
\end{defn}

\begin{lemma}\label{thm:Hermitian-operator}
  If $A = (a_{x,y})_{x,y \in X}$ is an infinite matrix, then matrix multiplication \eqref{eqn:matrix-mult} defines an operator 
  \linenopax
  \begin{align}\label{eqn:A-as-operator}
    T_A: c_0(X) \to \ell^2(X)
  \end{align}
  if and only if for any fixed $y \in X$, the function $x \mapsto a_{x,y}$ is in $\ell^2(X)$. In this case, $T_A$ is Hermitian if and only if $a_{x,y} = \cj{a_{y,x}}$ for all $x,y \in X$.
  \begin{proof}
    This is clear because $A \gd_y = a_{\cdot, y}$ is the column in $A$ with index $y$ and $\la \gd_x, A\gd_y\ra = a_{x,y}$. The latter claim is standard.
  \end{proof}
\end{lemma}

\begin{lemma}\label{thm:matrix-mult-for-adjoint}
  Let $A = (a_{x,y})_{x,y \in X}$ be an infinite matrix which defines an operator $T_A: c_0(X) \to \ell^2(X)$ as in Lemma~\ref{thm:Hermitian-operator}. Then the following two conditions are equivalent, for two vectors $v$ and $w$ in $\ell^2(X)$:
  \begin{enumerate}[(i)] 
    \item $w(y) = \sum_{x \in X} \cj{a_{x,y}} v(x)$ is absolutely convergent for each $y \in X$, and $w \in \ell^2(X)$.
    \item $v \in \dom(T_A^\ad)$ and $T_A^\ad v = w$.
  \end{enumerate}
  In particular, the action of the operator $T_A^\ad$ is given by formula \eqref{eqn:matrix-mult}.
  \begin{proof}
    (i) $\implies$ (ii). To show that $v \in \dom T_A^\ad$, note that $\la T_A u, v\ra_{\ell^2}$ is equal to 
    \linenopax
    \begin{align}
      \sum_{x \in X} \cj{T_A u(x)} v(x)
      &=\sum_{x \in X} \sum_{y \in Y} \cj{a_{x,y}} \cj{u(y)} v(x) 
      = \sum_{y \in X} \sum_{x \in X} \cj{a_{x,y}} v(x) \cj{u(y)} 
      = \sum_{y \in X} \cj{w(y)} u(x), \label{eqn:matrix-mult-for-adjoint:comp2}
    \end{align}
    by Fubini-Tonelli. This gives the estimate $\left|\la T_A u, v\ra\right| \leq \|w\|_{\ell^2} \|u\|_{\ell^2}$ by the Cauchy-Schwarz inequality, which means $v \in \dom T_A^\ad$. The equality $T_A^\ad v = w$ follows from \eqref{eqn:matrix-mult-for-adjoint:comp2}.
    
    For the converse, note that $w \in \ell^2(X)$ because $v \in \dom(T_A^\ad)$. Then the same calculation in reverse gives $\sum_{x \in X} \cj{u(x)} T_A^\ad v(x) = \sum_{x \in X} \sum_{y \in X} \cj{a_{x,y}} \cj{u(y)} v(x)$. 
  \end{proof}
\end{lemma}

\begin{cor}\label{thm:pointwise-adjoint}
  There exists $w \in \sH$ such that $\la v, T_A u\ra_{\ell^2} = \la w, u\ra_{\ell^2}$ holds for all $u \in \dom T_A $ if and only if $v \in \dom T_A^\ad$ and $T_A^\ad v = w$. If we additionally assume that $A$ is symmetric, the pointwise identity $(Av)(x) = w(x)$ holds for all $x \in X$.
\end{cor}

\subsection{Matrix Laplace operators on $\ell^2(X)$}
\label{sec:Matrices-and-self-adjoint-operators-in-L2}


In this section, we consider a Laplacian to be the operator associated to a matrix satisfying the conditions of Definition~\ref{def:matrix-Laplacian}. Our main result in this section is Theorem~\ref{thm:matrix-conds-for-essential-selfadjointness}, which asserts that these three elementary conditions are sufficient to ensure the associated operator is essentially self-adjoint, and hence has a well-defined and unique spectral representation.

\begin{defn}\label{def:matrix-Laplacian}
  If $X$ is a countably infinite set, then we say that the infinite matrix $A = (a_{x,y})_{x,y \in X}$ defines a \emph{(matrix) Laplacian} iff $A$ satisfies 
  \begin{enumerate}[(i)] 
    \item $a_{x,y} = {a_{y,x}}$, for all $x,y \in X$;
    \item $a_{x,y} \leq 0$ if $x \neq y$; and
    \item $\sum_{y \in X} a_{x,y} = 0$, for all $x \in X$.
  \end{enumerate} 
  In this case, we write $\Lap_A$ for the corresponding Hermitian operator $\Lap_A: c_0(X) \to \ell^2(X)$ defined by matrix multiplication, as in Lemma~\ref{thm:Hermitian-operator}. Note that it follows immediately from (ii)--(iii) that $a_{x,x} = -\sum_{y \in X\less\{x\}} a_{x,y} \geq 0$, for each $x \in X$, so the sum in (iii) is automatically absolutely convergent. 
\end{defn}
\begin{theorem}[Essential self-adjointness of matrix Laplacians on $\ell^2(X)$]
\label{thm:matrix-conds-for-essential-selfadjointness}
  If the infinite matrix $A = (a_{x,y})_{x,y \in X}$ defines a matrix Laplacian on $X$, then the corresponding Hermitian operator $\Lap_A: c_0(X) \to \ell^2(X)$ is essentially self-adjoint.
\end{theorem}

The proof of Theorem~\ref{thm:matrix-conds-for-essential-selfadjointness} requires Lemma~\ref{thm:<u,Au>=E(u)}, variants of which appear in the literature in different contexts, for example, \cite[Cor.~6.9]{Kig03} and \cite[Thm.~1.3.1]{FOT94}. Theorem~\ref{thm:matrix-conds-for-essential-selfadjointness} extends and corrects \cite[Thm.~3.1]{Jorgensen08} (the result is stated correctly, but there is an error in the proof). 

\begin{remark}\label{rem:general-measures-on-ell2}
  After a first version of this paper was completed, we discovered that Keller and Lenz have extended this result to the situation of more general measures in \cite{KellerLenz09} and \cite{KellerLenz10}, as long as the measure gives weight \iy to infinite paths. (This is true automatically for the counting measure, which we use exclusively). Note also that the results of \cite{KellerLenz09, KellerLenz10} allow for positive potentials (denoted therein by $c$). Consequently, one cannot hope to study the deficiency spaces of \Lap unless one considers (i) $\ell^2$ spaces with respect to a measure which violates this axiom, or (ii) some other Hilbert space entirely. In this paper, we elect to go with the latter option, and hence focus on the energy Hilbert space in \S\ref{sec:repkernel-energy-space}--\S\ref{sec:finite-approximants}. Related but less general results also appear in \cite{Web08, Woj07}; see also \cite{Woj09}. 
\end{remark}


\begin{lemma}[Semiboundedness of $\Lap_A$ on $\ell^2(X)$]\label{thm:<u,Au>=E(u)}
  If the infinite matrix $A = (a_{x,y})_{x,y \in X}$ defines a matrix Laplacian on a countably infinite set $X$, then $\Lap_A$ is semibounded and and positive semidefinite with
  \linenopax
  \begin{align}\label{eqn:<u,Au>=E(u)}
    \la u, \Lap_A u \ra_{\ell^2}
    = \tfrac12 \sum_{x,y \in X} (-a_{x,y}) |u(x) - u(y)|^2,
    \q\text{for all } u \in c_0(X).
  \end{align}
  \begin{proof}
    Note that the right-hand side of \eqref{eqn:<u,Au>=E(u)} is a sum of nonnegative terms by Definition~\ref{def:matrix-Laplacian}(ii), and that it is a finite sum by \eqref{eqn:def:c_0}. The double summation on the right-hand side of \eqref{eqn:<u,Au>=E(u)} is
    \linenopax
    \begin{align*}
      \sum_{x,y \in X} a_{x,y}|u(x) - u(y)|^2
      &= \sum_{x \in X}\sum_{y \in X} a_{x,y}|u(x)|^2 
       - 2 \sum_{x,y \in X} a_{x,y} \Re(\cj{u(x)}u(y)) 
       + \sum_{y \in X}\sum_{x \in X} a_{x,y}|u(y)|^2. 
    \end{align*}
    The last sum on the right side vanishes by Definition~\ref{def:matrix-Laplacian}(iii), and similarly the first sum vanishes by combining parts (i) and (iii) of the same definition.
    Thus, the computation continues as
    \linenopax
    \begin{align*}
      &= - 2 \sum_{x,y \in X} a_{x,y} \Re(\cj{u(x)}u(y)) 
      = - \sum_{x,y \in X} a_{x,y} \cj{u(x)}u(y) 
         - \sum_{x,y \in X} a_{y,x} \cj{u(x)}u(y) 
       = - 2\la u, \Lap_A u\ra_{\ell^2},
    \end{align*}
    which gives \eqref{eqn:<u,Au>=E(u)}. In view of assumption (i), we further get that $\la u, \Lap_A u\ra_{\ell^2} \geq 0$ for all $u \in c_0(X)$. Hence, the operator $\Lap_A$ is semibounded and positive semidefinite.
  \end{proof}
\end{lemma}

\begin{defn}\label{def:exhaustion-of-X}
  An \emph{exhaustion} of $X$ is a sequence of finite subsets $\{F_k\}_{k=1}^\iy$ satisfying $F_k \ci F_{k+1}$ and $X = \bigcup_{k=1}^\iy F_k$.
\end{defn}

We now return to the proof of Theorem~\ref{thm:matrix-conds-for-essential-selfadjointness}. 

\begin{proof}[Proof of Theorem~\ref{thm:matrix-conds-for-essential-selfadjointness}.]
  Assume that some $v \in \ell^2(X)$ satisfies 
  \linenopax
  \begin{align}\label{eqn:defect-identity}
    \sum_{y \in X} a_{x,y} v(y) = - v(x).
  \end{align}
  By applying Lemma~\ref{thm:matrix-mult-for-adjoint} and Lemma~\ref{thm:essentially-sa-iff-no-defect}, we must prove that $v=0$ to complete the proof of Theorem~\ref{thm:matrix-conds-for-essential-selfadjointness}. First, observe that (i)--(iii) imply that each of the following functions on $X \times X$ is summable, i.e., is in $\ell^1(X \times X)$:
  \linenopax
  \begin{align*}
    a_{x,y} |v(x)|^2, \; 
    a_{x,y} |v(y)|^2, \; 
    \cj{v(x)} a_{x,y} v(y), \; \text{ and }
    a_{x,y} |v(x)-v(y)|^2.
  \end{align*}
  Note that with (i)--(iii), Fubini's theorem applies to the double summations of each of these functions.  
Pick an exhaustion $\{F_k\}_{k=1}^\iy$ as in Definition~\ref{def:exhaustion-of-X}, and then \eqref{eqn:defect-identity} gives 
  \linenopax
  \begin{align}\label{eqn:l2-and-pw-convergence-of-v_k}
    \lim_{k \to \iy} \left\|v + \sum_{y \in F_k} a_{x,y} v(y)\right\|_{\ell^2} = 0,
    \qq\text{and}\qq
    \lim_{k \to \iy} \sum_{y \in F_k} a_{x,y} v(y) = -v(x), \, \forall x \in X.
  \end{align}
  The argument in the proof of Lemma~\ref{thm:<u,Au>=E(u)} now yields the following:
  \linenopax
  \begin{align}\label{eqn:sum-identity-on-XxFk}
    \sum_{x \in X} \sum_{y \in F_k} &(-a_{x,y}) |v(x) - v(y)|^2 
    &= 2 \sum_{x \in X} \sum_{y \in F_k} \cj{v(x)} a_{x,y} v(y) 
    - \sum_{x \in X} |v(x)|^2 \sum_{y \in F_k} a_{x,y}
    - \sum_{y \in F_k} |v(y)|^2 \sum_{x \in X} a_{x,y}.  
  \end{align}
  Combining (iii) with \eqref{eqn:l2-and-pw-convergence-of-v_k} and Fatou's lemma, we can pass to the limit in \eqref{eqn:sum-identity-on-XxFk}. To compute this limit, note that for the first term on the right-hand side in \eqref{eqn:sum-identity-on-XxFk}, equation \eqref{eqn:l2-and-pw-convergence-of-v_k} gives
  \linenopax
  \begin{align}
    \sum_{x \in X} \sum_{y \in F_k} \cj{v(x)} a_{x,y} v(y) 
    &= \sum_{x \in X} \cj{v(x)} \sum_{y \in F_k} a_{x,y} v(y) 
    \limas{k} -\sum_{x \in X} |v(x)|^2 = -\|v\|_{\ell^2}^2.  
  \end{align}
  The second term on the right-hand side in \eqref{eqn:sum-identity-on-XxFk} vanishes because $\lim_{k \to \iy} \sum_{y \in F_k} a_{x,y} = 0$, by (iii). Consequently, one obtains the identity
  \linenopax
  \begin{align}\label{eqn:minus-norm(v)}
    \sum_{x \in X} \sum_{y \in X} (-a_{x,y}) |v(x)-v(y)|^2
    = -2\|v\|_{\ell^2}^2.
  \end{align}
  Since the left-hand side in \eqref{eqn:minus-norm(v)} is nonnegative (as noted initially) and the right-hand side is nonpositive, it must be the case that $\|v\|_\ell^2 = 0$, whence $v=0$.
\end{proof}

For future use, we note the following corollary which follows easily from a known characterization of positive semidefinite infinite matrices.

\begin{cor}\label{thm:-A-is-positive-semidefinite}
  Suppose the infinite matrix $A = (a_{x,y})_{x,y \in X}$ defines a matrix Laplacian on $X$. If $\{F_k\}_{k=1}^\iy$ is an exhaustion of $X$ as in Definition~\ref{def:exhaustion-of-X}, and $A(F_k) := (a_{x,y})_{x,y \in F_k}$ is the finite submatrix of $A$ corresponding to $F_k$, then $\det A(F_k) \geq 0$ for every $k$.
\end{cor}

\section{Axioms for a reproducing kernel energy space}
\label{sec:repkernel-energy-space}

In this section, we give some axioms for a certain type of reproducing kernel Hilbert space that distill the essential properties of the energy space \HE discussed in \S\ref{sec:Matrices-and-self-adjoint-operators-in-the-energy-space}.

\subsection{The axioms}
\label{sec:axioms}
  Let us fix a set $X$ and suppose that we have a quadratic form \sQ defined for functions $u$ on $X$ with domain \smash{$\dom \sQ = \{u \suth \sQ(u) < \iy\}$}. Suppose that $\sH = \dom\sQ/\ker\sQ$ is a Hilbert space with respect to the inner product obtained from \sQ by polarization, that is, under
  \linenopax
  \begin{align}\label{eqn:axiomatic-inner-prod}
     \la u, v\ra_\sH := \tfrac14\left(\sQ(u+v) \vstr[2.2]- \sQ(u-v) +
     \ii\sQ(u+\ii v) - \ii\sQ(u-\ii v)\right),
  \end{align}
  and that \sH satisfies the following axioms.

\begin{axm}\label{axm:constants}
  The constant function $\one(x) \equiv 1$ is an element of $\ker\sQ$. 
\end{axm}

\begin{axm}\label{axm:Diracs}
  For each $x \in X$, the Dirac (point) mass $\gd_x$ is contained in $\dom \sQ$, where $\gd_x$ is defined by
  \linenopax
  \begin{align}\label{eqn:Diracs}
      \gd_x(y) = 
      \begin{cases}
        1, &y=x, \\
        0, &y \neq x. 
      \end{cases}
  \end{align}
\end{axm}


\begin{axm}\label{axm:kernel}
  For every pair of points $x,y \in X$, there is a constant $C = C_{x,y}$ such that
  \linenopax
  \begin{align}\label{eqn:axm:base-point}
    |f(x) - f(y)| \leq C \|f\|_\sH, \qq \text{for all } f \in \dom \sQ. 
  \end{align}
\end{axm}


\begin{remark}\label{rem:equivalence-class-notation}
  In most cases, it will not be necessary to distinguish between a function $u \in \dom \sQ$ and its corresponding equivalence class in \sH. However, whenever it is useful to make the distinction, we use the notation $[u]_\sQ$ to indicate the equivalence containing the function $u$ defined on $X$.
\end{remark}

\begin{remark}\label{rem:resistance-forms}
  The axiom system above is very similar to the notion of \emph{resistance form} as developed in \cite{Kig01,Kig03} (see also the references therein), although the axioms above evolved independently, and from different considerations. Axioms~\ref{axm:constants}--\ref{axm:kernel} allow for slightly more generality than resistance forms.
\end{remark}

\begin{defn}\label{def:relative-reproducing-kernel}
  For a Hilbert space \sH of functions on $X$, a \emph{reproducing kernel}  is a family $\{v_x\}_{x \in X} \ci \sH$ satisfying
  \begin{align}\label{eqn:reproducing-kernel}
    \la v_x, u\ra_\sH = u(x), \qq\text{for all } x \in X \text{ and for any } u \in \dom\sQ,
  \end{align}
  and a \emph{relative reproducing kernel} is a family $\{v_{x,y}\}_{x,y \in X} \ci \sH$ satisfying
  \begin{align}\label{eqn:axiomatic-repkernel}
    \la v_{x,y}, u\ra_\sH = u(x) - u(y), \qq\text{for all } x,y \in X\text{ and for any } u \in \dom\sQ.
  \end{align}
\end{defn}

\begin{lemma}\label{thm:axiomatic-repkernel}
  Axiom~\ref{axm:kernel} ensures the existence of a relative reproducing kernel for \sH.
  \begin{proof}
    First, note that $\la v_{x,y}, u\ra_\sH$ means $\la v_{x,y}, [u]_\sQ \ra_\sH$, as in Remark~\ref{rem:equivalence-class-notation}. Next, Axiom~\ref{axm:kernel} asserts continuity of the linear functional $L_{x,y} :\sH \to \bC$ defined by $L_{x,y} u = u(x) - u(y)$, so Riesz's lemma gives a $v_{x,y} \in \sH$ satisfying \eqref{eqn:axiomatic-repkernel}, for each $x,y \in X$.
  \end{proof}
\end{lemma}

Henceforth, it will be convenient to fix a reference point $o \in X$ to act as an origin and consider the singly-indexed family $\{v_x\}_{x \in X} \ci \sH$, where $v_{x} = v_{x,o}$. All results will be independent of the choice of $o$. 


\begin{defn}\label{def:Laplacian}
  Define the possibly unbounded \emph{(abstract) Laplace operator} with domain
  \linenopax
  \begin{equation}\label{eqn:Laplacian-domain}
    \dom \Lap := \spn\{\one, \{v_x\}_{x \in X \less \{o\}}\} \ci \sH,
  \end{equation}
  by the pointwise equation
  \linenopax
  \begin{equation}\label{eqn:Laplacian-ptwise}
    (\Lap w)(x) := \ipH{\gd_x,w}.
  \end{equation}
\end{defn}

\begin{remark}\label{rem:Lap-equivs}
   In \eqref{eqn:Laplacian-ptwise}, the notation $\ipH{\gd_x,w}$ really means $\ipH{[\gd_x]_\sQ,w}$, but we can suppress the equivalence class notation because any two representatives differ by an element of $\ker \sQ$.
\end{remark}

\begin{cor}\label{thm:v_x-dense}
  $\dom \Lap$ is dense in \sH. 
  \begin{proof}
    Suppose that $\ipH{v_x,u} = 0$ for all $x \in X$. Then by \eqref{eqn:axiomatic-repkernel}, $u$ must be constant.
  \end{proof}
\end{cor}

\begin{remark}\label{rem:motivation-for-axm:base-point}
  It is often the case that $\Lap w = \gd_x$ does not have a solution in \sH (this is explored in \cite{DGG}. However, $\Lap w = \gd_x - \gd_o$ always has a solution; this follows from Lemma~\ref{thm:vx-is-a-dipole}, just below, and is due in some sense to the ``balanced'' nature of $\gd_x-\gd_y$; see \cite[\S{III}.3]{Soardi94}. For either $\Lap w = \gd_x$ or $\Lap w = \gd_x - \gd_o$, the solution $w$ is nonunique precisely when $\ker \Lap \cap \sH$ is nontrivial.
\end{remark}

\begin{lemma}\label{thm:vx-is-a-dipole}
  For each $x \neq o$, one has $\Lap v_x(y) = (\gd_x - \gd_o)(y)$, for all $y \in X$.
  \begin{proof}
    From \eqref{eqn:Laplacian-ptwise}, we have $\Lap v_x(y) = \ipH{[\gd_y]_\sQ,v_x}$, where $[\gd_y]_\sQ \in \sH$ is the class containing the function $\gd_y$ defined as in \eqref{eqn:Diracs}; see also Remark~\ref{rem:equivalence-class-notation} and Remark~\ref{rem:Lap-equivs}. The result now follows via \eqref{eqn:axiomatic-repkernel}  by 
    \linenopax
    \begin{align*}
      \ipH{[\gd_y]_\sQ,v_x} 
      &= \gd_y(x) - \gd_y(o) = \gd_x(y) - \gd_o(y). 
      \qedhere
    \end{align*}
  \end{proof}
\end{lemma}

\begin{remark}\label{rem:Lap-in-product}
  From Lemma~\ref{thm:vx-is-a-dipole}, Axiom~\ref{axm:Diracs} implies that $\Lap u \in \dom \sQ$ and hence $\Lap u$ represents a unique element of \sH. 
  Thus, expressions like $\ipH{u, \Lap v}$ are well-defined, and in particular, so is $\ipH{u, \Lap v_x}$ for any $x \in X$, if $u \in \sH$ or $u \in \dom\sQ$.
\end{remark}

The following lemma was suggested by (and due to) the referee, for its use in Lemma~\ref{thm:M(x,y)-is-pd}.

\begin{lemma}\label{thm:linear-independence-of-kernel}
  Under Axioms~\ref{axm:constants}--\ref{axm:kernel}, the set $\{v_x\}_{x \in X \less \{o\}}$ is linearly independent.
  \begin{proof}
    Suppose one has a linear combination $u = \sum \gx_x v_x = 0$ where at most finitely many of the coefficients $\gx_x$ are nonzero. Then $u \in \dom\Lap$ and Lemma~\ref{thm:vx-is-a-dipole} gives
    \begin{align*}
      0 = \Lap u = \sum \gx_x \Lap v_x = \sum \gx_x(\gd_x-\gd_o),
    \end{align*}
    whence $\gx_x=0$ for all $x \in X\less \{o\}$.
  \end{proof}
\end{lemma}

\subsection{Some basic properties of the abstract Laplacian}
\label{sec:two-basic-lemmas}

In this section, we show that the definitions given above are sufficient to prove that \Lap is Hermitian and even semibounded. Throughout this section, we abuse notation as described in Remark~\ref{rem:Lap-equivs} and denote both a function and the equivalence class containing it by the same symbol.

\begin{lemma}\label{thm:lap-is-almost-Green}
 If $\gd_{xy}$ is the Kronecker delta, then 
  \begin{equation}\label{eqn:<vx,Lap(vy)>=gd(xy)+1}
    \ipH{v_x,\Lap v_y} = \gd_{xy} + 1 - \gd_{xo} - \gd_{yo},
    \qq \forall x,y \in X.
  \end{equation}
  \begin{proof}
    Note that $\Lap v_y \in \sH$ by Remark~\ref{rem:Lap-in-product}, and so 
    \linenopax
    \begin{align*}
      \ipH{v_x,\Lap v_y}
      &= (\Lap v_y)(x) - (\Lap v_y)(o) 
      = \ipH{\gd_x,v_y} - \ipH{\gd_o,v_y},
    \end{align*}
    by \eqref{eqn:axiomatic-repkernel} and \eqref{eqn:Laplacian-ptwise}.
    Again using \eqref{eqn:axiomatic-repkernel}, the result follows via
    \linenopax
    \begin{align*}
      \ipH{\gd_x,v_y} - \ipH{\gd_o,v_y} 
      &= (\gd_x(y) - \gd_x(o)) - (\gd_o(y) - \gd_o(o))
      = \gd_{xy} + 1 - \gd_{xo} - \gd_{yo}.
      \qedhere
    \end{align*}
  \end{proof}
\end{lemma}

\begin{lemma}\label{thm:lap-is-hermitian}
  The operator \Lap is Hermitian on \sH.
  \begin{proof}
    Note that \eqref{eqn:<vx,Lap(vy)>=gd(xy)+1} is symmetric in $x$ and $y$, and \bR-valued. Thus 
    \linenopax
    \begin{align*}
      \ipH{\Lap v_x, v_y} 
       &= \cj{\ipH{v_y, \Lap v_x}}
       = \cj{\gd_{yx} + 1- \gd_{xo} - \gd_{yo}}
       = \gd_{xy} + 1 - \gd_{xo} - \gd_{yo}
       = \ipH{v_x, \Lap v_y}.
       \qedhere
    \end{align*}  
  \end{proof}
\end{lemma}

\begin{lemma}\label{thm:invariance-of-kerQ-under-Lap}
  The action of \Lap on $\dom \sQ$ passes to the quotient: $[\Lap u]_\sQ = \Lap [u]_\sQ$ for any $u \in \dom \sQ$.
  \begin{proof}
    This is equivalent to showing that the kernel of \sQ is invariant under the action of \Lap.
    Suppose that $\gy = \sum_{z \in F} \gx_z v_z$ is an element of $\ker\sQ$, and that $F$ is finite. Then $\ipH{\gy,\gf}=0$ for every $\gf \in \sH$, so with $\gf = \Lap v_x$ (which is well-defined by Remark~\ref{rem:Lap-in-product}), Lemma~\ref{thm:lap-is-hermitian} gives $0 = \ipH{\gy,\Lap v_x} = \ipH{\Lap\gy, v_x}$, for every $x \in X$, so that $\Lap\gy \in \ker \sQ$ by Corollary~\ref{thm:v_x-dense}. The conclusion follows.
  \end{proof}
\end{lemma}

\begin{lemma}\label{thm:semibounded}
   The operator \Lap given in Definition~\ref{def:Laplacian} is semibounded as in Definition~\ref{def:semibounded}. 
  \begin{proof}
    If $u \in \dom\Lap$, then $u = \sum_{x \in F} \gx_x v_x$ for some finite $F \ci X\less\{o\}$ by \eqref{eqn:Laplacian-domain} and 
    \linenopax
    \begin{align}\label{eqn:<u,Lapu>=2sums}
      \ipH{u, \Lap u}
      &= \sum_{x,y \in F} \cj{\gx_x} \gx_y \ipH{v_x,\Lap v_y} 
       = \sum_{x,y \in F} \cj{\gx_x} \gx_y (\gd_{xy} +  1)
       = \sum_{x \in F} |\gx_x|^2 + \left|\sum_{x \in F} \gx_x \right|^2 \geq 0,
    \end{align} 
    by Lemma~\ref{thm:lap-is-almost-Green}.
  \end{proof}
\end{lemma}

\begin{remark}\label{rem:semibounded-implies-pd}
  In fact, one can draw a much stronger conclusion than just semiboundedness from Lemma~\ref{thm:semibounded}: note from the proof that $\ipH{u, \Lap u} = 0$ implies $\sum |\gx_x|^2 = 0$ and thus $u=0$. 
\end{remark}

\begin{lemma}\label{thm:lap-acts-on-ipH}
  Fix $y \in X$ and consider $\gf(x) := \ipH{v_x, v_y}$ as a function of $x$ on $X$. Let $\Lap_x$ denote the application of \Lap with respect to the $x$ variable. Then 
  \linenopax
  \begin{align}\label{eqn:lap-acts-on-ipH}
    \Lap_x \ipH{v_x, v_y} = \ipH{\Lap v_x, v_y} + \gd_{xo} - 1.
  \end{align}
  \begin{proof}
    Note that $\gf(x) = v_y(x)-v_y(o)$ for each fixed $y$, so that $\gf=v_y$ in \sH. Then $\gf \in \dom\Lap$ and
    \linenopax
    \begin{align*}
      \Lap_x \ipH{v_x, v_y}
      &= \Lap_x(v_y(x)-v_y(o))
      =\Lap_x v_y(x)
      = \gd_y(x) - \gd_y(o)
      = \gd_{xy} - \gd_{yo}.
    \end{align*}
    Now \eqref{eqn:lap-acts-on-ipH} follows by Lemma~\ref{thm:lap-is-almost-Green} and Lemma~\ref{thm:lap-is-hermitian}.
  \end{proof}
\end{lemma}

The authors are grateful to the referee for suggesting the above streamlined version of the proof.

\subsection{Foundations of reproducing kernel Hilbert spaces}
\label{sec:foundations-of-RKHS}

This subsection aims to give some brief historical context for \S\ref{sec:repkernel-energy-space} in general, and Lemma~\ref{thm:M(x,y)-is-pd} in particular. 

\begin{defn}\label{def:positive-semidefinite}
  One says $M:X \times X \to \bC$ is a \emph{positive semidefinite (psd)} function iff \linenopax
  \begin{align}\label{eqn:def:positive-semidefinite}
    \sum_{x \in F} \cj{\gx_x} M(x,y) \gx_y \geq 0, \q \forall \gx=\{\gx_x\}_{x \in X},
  \end{align}
  whenever $F \ci X$ is finite. Informally, we describe this condition by saying ``$M$ is psd on $X$''. Similarly, one says $M:X \times X \to \bC$ is \emph{positive definite (pd)} iff the inequality in \eqref{eqn:def:positive-semidefinite} is strict for all finitely supported nonzero sequences $c$.
\end{defn}

The theory of positive (semi)definite functions is broad and powerful (see, e.g. \cite{Ber84}) but we are interested primarily in two closely related theorems stemming from the work of von Neumann and Kolmogorov. The first one (Theorem~\ref{thm:vNeu-construction}) is a generalization and amalgamation of some results of \cite[\S5--6]{OTERN}. The second one (Theorem~\ref{thm:Kolmo-construction}) adds the slightly stronger hypothesis of pd (instead of psd) and is able to draw a \emph{much} stronger conclusion: one is able to produce a \emph{Gaussian} measure on the resulting space.
The following result is the foundation for the study of reproducing kernel Hilbert spaces as developed by Aronszajn \cite{Aronszajn50} and \cite{PaSc72}.

\begin{theorem}\label{thm:vNeu-construction}
  Given a psd function $M$ on $X$, there exists a Hilbert space \sH with an inner product $\ipH{\cdot, \cdot}$ and a function $v:X \to \sH$ such that
  \begin{enumerate}[(i)]
    \item \label{itm:M(x,y)=<vx,vy>}
      $M(x,y) = \ipH{v_x,v_y}$ for all $x,y \in \sH$, and
    \item \label{itm:v(x)-dense-in-H}
      $\cl \spn{v_x} = \sH$.
  \end{enumerate}
  Moreover, $v:X \to \sH$ is unique up to unitary equivalence when (i) and (ii) are satisfied. In fact, $v_x$ is defined to be the equivalence class of $M(\cdot,x)$ under a certain quotient map.
\begin{proof}[Sketch of proof.]
  The vector space of all finite linear combinations $\sum \gx_x M(\cdot,x)$ can be made into a pre-Hilbert space by defining the sesquilinear form
  \linenopax
  \begin{align*}
    \ipH[{\negsp[10]}M]{\sum_{x \in F} a_x M(\cdot,x),\sum_{y \in F} b_y M(\cdot,y)}
    := \sum_{x,y \in F} \cj{a_x} M(x,y) b_y,
  \end{align*}
  where $F$ is a finite subset of $X$ containing the support of $a$ and $b$.
  One can verify that this satisfies a generalized Cauchy-Schwarz inequality, and one can therefore obtain a Hilbert space by modding out by the kernel of $M$ and taking the completion.
\end{proof}
\end{theorem}

Theorem~\ref{thm:Kolmo-construction} is an alternative approach to this construction (see\cite{PaSc72}) which allows one to realize the Hilbert space \sH of Theorem~\ref{thm:vNeu-construction} as $L^2(\gW,\prob)$. This version is more probabilistic in flavour; in fact, Kolmogorov's consistency construction is lurking in the background.

\begin{theorem}\label{thm:Kolmo-construction}
  Given a psd function $M$ on $X \times X$, there exists a probability space $(\gW,\prob)$ and a collection of random variables $\{\rvar_x\}_{x \in X}$ such that for all $x,y \in X$,
  \linenopax
  \begin{align}\label{eqn:random-var-mean-and-correlation}
    \Ex(\rvar_x) = 0 \q\text{and}\q \Ex(\rvar_x,\rvar_y) = M(x,y).
  \end{align}
  Moreover, if $M$ is pd, then \prob can be taken to be Gaussian.
\end{theorem}

Lemma~\ref{thm:M(x,y)-is-pd} can be considered as a (somewhat trivial) converse of Aronszajn's theorem, 
and will be useful in \S\ref{sec:finite-approximants}.

\begin{lemma}\label{thm:M(x,y)-is-pd}
  Given any function $v:X \to \sH$ mapping $X$ into a Hilbert space, the function defined by 
  $  M(x,y) := \ipH{v(x),v(y)}$
  is pd on $X$. 
  \begin{proof}
    If $\gx=\{\gx_x\}_{x \in X}$ is not identically $0$, then for any finite $F \ci X$,
    \linenopax
    \begin{align}\label{eqn:<v_x,v_y>-is-psd}
      \sum_{x \in F} \sum_{y \in F} \cj{\gx_x} \gx_y \ipH{v(x),v(y)}
	  = \ipH[{\negsp[10]}\sH]{\sum_{x \in F} \gx_x v(x), \sum_{y \in F} \gx_y v(y)}
	  = \|w\|_\sH^2 > 0,
	\end{align}
	where $w \in \sH$ is the function defined by $w = \sum_{x \in F} \gx_x v(x)$. Note that the final inequality is strict by Lemma~\ref{thm:linear-independence-of-kernel}. 
  \end{proof}
\end{lemma}


\section{The Laplacian as an operator on the energy space}
\label{sec:Matrices-and-self-adjoint-operators-in-the-energy-space}

In this section, we introduce the setting of a resistance network $(\Graph,\cond)$. There are a couple of different (but very natural) Hilbert spaces of functions defined on such a domain, both of which are important for understanding the underlying network. The study of a network is inextricably linked to the study of the associated Laplace operator: note that if $A$ is the adjacency matrix of a network, then as matrices, $\Lap = \cond \id - A$; see \eqref{eqn:pointwise-Lap-in-HE-as-sum}. 

This section aims to compare the $\ell^2(\Graph)$ theory of \Lap (as discussed in \S\ref{sec:self-adjoint-operators-in-L2}) with the behavior of \Lap on a second Hilbert space of functions naturally associated to the network: the energy space \HE; see \cite{DGG,ERM,bdG,RBIN} and also the references \cite{Kig01,Kig03,Lyons:ProbOnTrees,Soardi94}.%
  \footnote{\HE is different from the space $\mathbf{D}$ discussed in \cite{Soardi94} (also called $(\sE,\sF_V)$ in \cite[Prop.~2.19]{Kig03}), but the two are closely related; see \cite[\S4.1]{DGG}, for example.} 
It is defined in Lemma~\ref{def:HE} from an energy form \energy on functions on $(\Graph,\cond)$ defined in Definition~\ref{def:energy-form}.   

The results of \S\ref{sec:self-adjoint-operators-in-L2} imply that the network Laplacian is essentially self-adjoint as an operator on $\ell^2(\Graph)$, i.e., on $\ell^2(\Graph,\gm)$ where \gm is counting measure; see also \cite{KellerLenz09, KellerLenz10}. However, the action of the Laplacian on \HE is markedly different. In particular, it is not always essentially self-adjoint as an operator on \HE, in sharp contrast to Theorem~\ref{thm:matrix-conds-for-essential-selfadjointness}. Example~\ref{exm:unbounded-B-matrix-defect} illustrates this phenomenon with an explicitly computed defect eigenvector and (nonzero) deficiency indices.

It also turns out that there is no natural onb for \HE; the natural candidate would be the Dirac masses $\{\gd_x\}_{x \in \Graph}$, but these are not orthogonal and typically don't even have dense span in \HE. Consequently, we rely on a reproducing kernel $\{v_x\}_{x \in \Graph \less\{o\}}$, as developed axiomatically in the previous section. In fact, this is part of the motivation behind \S\ref{sec:repkernel-energy-space}.

Due in part to their close relation with Markov chains, there is a massive literature on resistance networks (not always using this terminology). Many studies use Hilbert space techniques, but almost all of these focus on $\ell^2(\Graph,\gm)$; see \cite{Soardi94, Chu01} and the references therein; other articles use methods from potential theory and discrete harmonic analysis \cite{Kig01,Kig03}. See also \cite[\S9]{Lyons:ProbOnTrees} for an alternative view on the energy space, presented in terms of an $\ell^2$ space of functions on the \emph{edges} of $G$.

\subsection{Networks and the energy space}

\begin{defn}\label{def:resistance-networks}
  A \emph{resistance network} is a connected weighted graph $(\Graph,\cond)$. Here $\Graph = (\verts,\edges)$ is a graph with a countable vertex set \verts, and at most one edge $e \in \edges$ between any two vertices. From this point onward, we write $x \in \Graph$ to indicate that $x \in \verts$. The adjacency relation on \Graph is determined entirely by the conductance function $\cond:\verts \times \verts \to [0,\iy)$, a nonnegative and symmetric real-valued function denoted $\cond_{xy} = \cond(x,y)$. We say $x,y \in \Graph$ are connected by an edge of weight $\cond_{xy}$ if and only if $\cond_{xy} > 0$; in this case, we write $x \nbr y$. Vertices may not have finite valency, but they must have finite total conductance: 
  \linenopax
  \begin{align}
    \cond(x) := \sum_{y \in \Graph} \cond_{xy} < \iy. 
  \end{align}
  We also assume $\cond_{xx}=0$ for every $x \in \Graph$. 
\end{defn}
In Definition~\ref{def:resistance-networks}, the term \emph{connected} means that for all $x,y \in \Graph$, there is a finite sequence $\{x=x_0, x_1, \dots, x_n=y\} \ci X$ such that $\cond_{x_i x_{i-1}} > 0$ for $i = 1,\dots, n$. There is a bijective correspondence between the class of resistance networks and the class of irreducible reversible Markov chains; the correspondence is given by considering \verts as the state space and defining the transition probability by $p(x,y) = \cond_{xy}/\cond(x)$, for vertices (states) $x$ and $y$. 
  
\begin{defn}\label{def:energy-form}
  For functions $u,v$ on a resistance network, one can define the (sesquilinear) \emph{energy form}
  \linenopax
  \begin{align}\label{eqn:def:energy-form}
    \energy (u,v) = \frac12 \sum_{x,y \in \Graph} \cond_{xy}(\cj{u(x)} - \cj{u(y)}) (v(x)-v(y))
  \end{align}
  with domain
     $\dom \energy := \{u :\Graph \to \bC \suth \energy(u,u) < \iy\}$.
  One says that $\energy(u) := \energy(u,u)$ is the \emph{energy} of $u$. 
\end{defn}

It is clear from \eqref{eqn:def:energy-form} and the connectedness of the network that $\energy(u) = 0$ iff $u$ is constant, so $\ker \energy = \bC\one$. Therefore, we define an equivalence relation by $u \sim v$ iff $u(x)-v(x)=k$ for some fixed $k \in \bC$. 

\begin{lemma}\label{thm:energy-space-is-complete}
  Under the above equivalence relation, and with 
  $\|\cdot\|_\energy = \sqrt{\energy(\cdot,\cdot)}$, the quotient
  \linenopax
  \begin{align}\label{eqn:matrix-energy-space}
    \HE := \frac{\dom \energy}{\ker \energy} 
         = \{u + \bC\one \suth u:\Graph \to \bC \text{ and } \|u\|_\energy < \iy\}
  \end{align}
  is a Hilbert space, and the elements of \HE are functions on $\Graph$ modulo constants.
  \begin{proof}
    It can be checked directly that the above collection of (equivalence classes of) functions on \Graph is complete via an isometric embedding into a larger Hilbert space as in \cite{Lyons:ProbOnTrees,MuYaYo} or by a standard Fatou's lemma argument as in \cite{Soardi94}. 
  \end{proof}
\end{lemma}

\begin{defn}\label{def:HE}
  The \emph{energy space} is the Hilbert space \HE with inner product $\la u,v\ra_\energy := \energy(u,v)$.
\end{defn}

\begin{theorem}\label{thm:energy-via-axioms}
  The energy space is a special case of the axiomatic presentation in \S\ref{sec:repkernel-energy-space}.
  \begin{proof}
  Note that \eqref{eqn:def:energy-form} gives 
  \linenopax
  \begin{align}\label{eqn:}
    \ipH[\energy]{\gd_x,\gd_y} = -\cond_{xy},
    \qq\text{and}\qq
    \ipH[\energy]{\gd_x,\gd_x} = \cond(x).    
  \end{align}
  In particular, the condition $\cond(x) < \iy$ ensures $\gd_x \in \HE$ for every $x \in \Graph$, and so Axiom~\ref{axm:Diracs} is satisfied. To see that Axiom~\ref{axm:kernel} is satisfied, one can argue as in \cite[Lem.~2.4]{DGG}: since \Graph is connected, choose a path $\{x_i\}_{i=0}^n$ with $x_0=y$, $x_n=x$ and $\cond_{x_i,x_{i-1}}>0$ for $i=1,\dots,n$, and the Schwarz inequality yields
    \linenopax
    \begin{align*}
      |L_{x,y} u |^2
      = |u(x)-u(y)|^2
      &= \left|\sum_{i=1}^n \sqrt{\frac{\cond_{x_i,x_{i-1}}}{\cond_{x_i,x_{i-1}}}} (u(x_i)-u(x_{i-1}))\right|^2 
      \leq k^2 \energy(u),
      \q\text{for}\q
      k = \left(\sum_{i=1}^n \cond_{x_i,x_{i-1}}^{-1} \right)^{1/2}.
  \end{align*}
  Consequently Lemma~\ref{thm:axiomatic-repkernel} applies and we have a relative reproducing kernel $\{v_x\}_{x \in \Graph} \in \HE$, as in Definition~\ref{def:relative-reproducing-kernel}, given by $v_x := v_{x,o}$. Although the elements of \HE are equivalence classes, computations can be performed using representatives whenever these computations are independent of the choice of representative. Abusing notation, we may take the function $u$ to be the representative of $u \in \HE$ satisfying $u(o)=0$.%
    \footnote{After an initial draft of this paper was complete, we discovered that researchers studying metrized graphs use a similar object; in \cite{BakerRumely, BakerFaber} this is called the ``$j$-function'' and is roughly given by $j_z(x,y) = v_{y,z}(x)$. The two objects do not precisely coincide because for metrized graphs, $x,y,z$ may be points in the interior of a edge, as edges are isometric to intervals in that context.} 
  \end{proof}
\end{theorem}

\begin{remark}\label{rem:Aronszajn-HE}
  Since one may add a constant function without changing the energy, $\dom \energy = \sH \oplus \bC$. Then, as in  \cite[Ex.~9.6(b)]{Lyons:ProbOnTrees}, one has
  \linenopax
  \begin{align}\label{eqn:quotient}
    \HE = \frac{\sH \oplus \bC}\bC.
  \end{align}
  Upon combining Definition~\ref{def:Laplacian} with \eqref{eqn:def:energy-form}, one obtains the Laplacian as the (graph) closure of the operator defined pointwise on the dense domain $\dom \Lap = \spn\{\one, \{v_x\}_{x \in X \less \{o\}}\}$ by
  \linenopax
  \begin{align}\label{eqn:pointwise-Lap-in-HE-as-sum}
    (\Lap u)(x) = \sum_{y \nbr x} \cond_{xy} (u(x) - u(y)).
  \end{align}
\end{remark}

\begin{remark}[The meaning of $\Lap u$]
  \label{rem:range-of-Lap}
  Note that $\Lap u$ is a function on \Graph, not an equivalence class of functions (the differences in \eqref{eqn:pointwise-Lap-in-HE-as-sum} specify the value of $\Lap u(x)$ unambiguously). 
  
  It is also clear that \Lap is Hermitian on \HE; note that Corollary~\ref{thm:lap-is-hermitian} holds in this context. It is also the case that \Lap commutes with conjugation, and this ensures that the deficiency indices of \Lap on \HE will be equal. \S\ref{sec:Lap-fails-to-be-essentially-self-adjoint-on-HE} discusses a situation in which \Lap on \HE has deficiency indices $(1,1)$.
  %
  
  Using the standard onb $\{\gd_x\}_{x \in \Graph}$ for $\ell^2(\Graph)$, and the matrix $A$ with entries $a_{x,y} = - \cond_{xy}$, formula \eqref{eqn:pointwise-Lap-in-HE-as-sum} is equivalent to matrix multiplication:
  \linenopax
  \begin{align}\label{eqn:Lap-as-matrix-multiplication}
    \Lap u = A u,
  \end{align}
  so that $a_{x,y}$ defines a matrix Laplacian on $\ell^2(X)$ in the sense of Definition~\ref{def:matrix-Laplacian}. In fact, the only real difference between Definition~\ref{def:resistance-networks} and Definition~\ref{def:matrix-Laplacian} is the addition of the connectedness condition, which appears in this section to ensure that the kernel of the energy form contains only (globally) constant functions. 

\end{remark}

\subsection{The Laplacian can fail to be essentially self-adjoint on \HE}
\label{sec:Lap-fails-to-be-essentially-self-adjoint-on-HE}

\begin{exm}[The geometric integers]
  \label{exm:unbounded-B-matrix-defect}
  For $b>1$, consider the network $(\bZ_+,b^n)$ consisting of the nonnegative integers with an edge of conductance $b^k$ connecting the vertex $k-1$ to the vertex $k$:
  \linenopax
    \begin{align*}
    \xymatrix{
      \vertex{0} \ar@{-}[r]^{b} 
      & \vertex{1} \ar@{-}[r]^{b^2} 
      & \vertex{2} \ar@{-}[r]^{b^3} 
      & \vertex{3} \ar@{-}[r]^{b^4} & \dots
    }
  \end{align*} 
  See \cite{RBIN,DGG,ERM,bdG,OTERN}.
\end{exm}
  
\begin{prop}[Defect on the geometric integers]
  As an operator on the energy space of the network $(\bZ_+,b^n)$, the Laplacian is not essentially-self-adjoint.
  \begin{proof}
  We will explicitly construct a function $u$ which has finite energy and which satisfies $\Lap u(n) = -u(n)$ at every vertex $n$ in the network. To this end, recursively define a system of polynomials $\{\gf_n\}$ and $\{\gy_n\}$ in the variable $r$ by
  \linenopax
  \begin{align}\label{eqn:poly-system}
    \left[\begin{array}{cc}
      \gf_n \\ \gy_n
    \end{array}\right] 
    =
    \left[\begin{array}{cc}
      1 & 1 \\ r^n & 1+r^n
    \end{array}\right] 
    \cdots
    \left[\begin{array}{cc}
      1 & 1 \\ r^2 & 1+r^2
    \end{array}\right] 
    \left[\begin{array}{cc}
      1 & 1 \\ r & 1+r
    \end{array}\right] 
    \left[\begin{array}{cc}
      0 \\ 1
    \end{array}\right] 
  \end{align}
  Putting $r=\frac1b$, the desired function $u$ is defined by $u(n) := \gy_n(1/b)$. Note that $\gf_{n} = \gf_{n-1} + \gy_{n-1}$ and \smash{$\gy_{n} = \gy_{n-1} + r^n \gf_{n}$}.
  Hence 
  \begin{align*}
    u(n) - u(n-1) 
    = \gy_n(\tfrac1b) - \gy_{n-1}(\tfrac1b)
    = \gy_{n-1}(\tfrac1b) + r^n \gf_n(\tfrac1b) - \gy_{n-1}(\tfrac1b)
    = r^n \gf_n(\tfrac1b)
  \end{align*}
  and therefore, suppressing the evaluation at the fixed value $r=1/b$, we have
  \begin{align}\label{eqn:phi-as-increment}
    \gf_n = b^n (u(n) - u(n-1)) 
  \end{align}
  Consequently, $\Lap u(n) = \gf_n - \gf_{n+1} = -\gy_n = -u(n)$ implies that $\Lap u = -u$. The proof will be complete once we show that $u \in \HE$, which is carried out in Lemma~\ref{thm:defect-vector-has-finite-energy}. 
  \end{proof}
\end{prop}

We will need the following lemma for the proof of Lemma~\ref{thm:defect-vector-has-finite-energy}.

\begin{lemma}\label{thm:geometric-bounds-on-p,q}
  There is an $m$ such that 
    \linenopax
    \begin{align}\label{eqn:phi-bound}
      \gf_{n} &\leq n^m, \qq \text{and} \qq
      \gy_{n} \leq (n+1)^m - n^m
      \qq\text{ for all } n \in \bZ_+,
    \end{align}
  where $\gf_n$ and $\gy_n$ are the polynomials defined in \eqref{eqn:poly-system}.   
  \begin{proof}
    We prove both bounds simultaneously by induction, so assume both bounds of \eqref{eqn:phi-bound} hold for $n-1$.
    The estimate for $\gf_{n} = \gf_{n-1} + \gy_{n-1}$ is immediate from the inductive hypotheses.
    For the $\gy_{n}$ estimate, choose an integer $m$ so that 
    \linenopax
    \begin{align*}
      m(m-1) 
      \geq \max\{t^2 r^t \suth t \geq 0\} 
      = \left(\frac{2}{e \log b}\right)^2. 
    \end{align*}
    Then $n^2 r^{n} \leq m(m-1)$ for all $n$, so
    \linenopax
    \begin{align*}
      2 + r^{n}
      \leq 2 + \frac{m(m-1)}{n^2}
      \leq \left(\frac{n-1}{n}\right)^m + \left(\frac{n+1}{n}\right)^m,
    \end{align*}
    by using the binomial theorem to expand $\left(\frac{n \pm 1}{n}\right)^m = \left(1 \pm \frac{1}{n}\right)^m$. Multiplying by $n^m$ gives
    \linenopax
    \begin{align*}
      \left(n^m - (n-1)^m\right) + r^{n} n^m &\leq (n+1)^m - n^m,
    \end{align*}
    which is sufficient because the left side is an upper bound for $\gy_{n} = \gy_{n-1} + r^{n} \gf_{n}$.
  \end{proof}
\end{lemma}

\begin{lemma}\label{thm:defect-vector-has-finite-energy}
  The defect vector $u(n) := \gy_n(\frac1b)$ has finite energy and is bounded. 
  \begin{proof}
    Applying Lemma~\ref{thm:geometric-bounds-on-p,q} to the formula for \energy yields 
    \linenopax
    \begin{align*}
      \energy(u)
      &= \sum_{n=1}^\iy b^n (u(n)-u(n-1))^2
      = \sum_{n=1}^\iy r^n \gf_n^2
      \leq \sum_{n=1}^\iy r^n n^{2m}
      = \operatorname{Li}_{-2m}(r) < \iy,
    \end{align*}
    since a polylogarithm indexed by a negative integer is continuous on $\bR$, except for a single pole at 1 (but recall that $r \in (0,1)$).
    
    To see that $u$ is bounded, combine \eqref{eqn:phi-as-increment} and  \eqref{eqn:phi-bound} to obtain $b^n \left(u(n)-u(n-1) \right) \leq n^m$, for some fixed $m$, whence the sequence of increments is summable in much the same way:
    \linenopax
    \begin{align*}
      \lim_{n \to \iy} u(n) - u(0)
      &= \sum_{n=1}^\iy \left(u(n)\vstr[2.2]-u(n-1)\right)
      \leq \sum_{n=1}^\iy r^n n^m
      < \iy.
      \qedhere 
    \end{align*}
  \end{proof}
\end{lemma}

  Lemma~\ref{thm:defect-vector-has-finite-energy} ensures that the defect vector is bounded; in the example in Figure~\ref{fig:defect-vector}, the defect vector has a limiting value of $\approx 4.04468281$, although the function value does not exceed 4 until $x=10$. The first few values of the function are
  \linenopax
  \begin{align*}
    u &= \left[ \tfrac{3}{2}, \tfrac{17}{8}, \tfrac{173}{64}, \tfrac{3237}{1024}, \tfrac{114325}{32768}, \tfrac{7774837}{2097152}, \tfrac{1032268341}{268435456}, \tfrac{270040381877}{68719476736}, \tfrac{140010315667637}{35184372088832}, \dots\right] \\
      &\approx \left[{1.5}, {2.125}, {2.7031}, {3.1611}, {3.4889}, {3.7073}, {3.8455}, \
{3.9296}, {3.9793}, {4.0080}, \dots \right]
  \end{align*}
  
  \begin{figure}
    \centering
    \scalebox{1.0}{\includegraphics{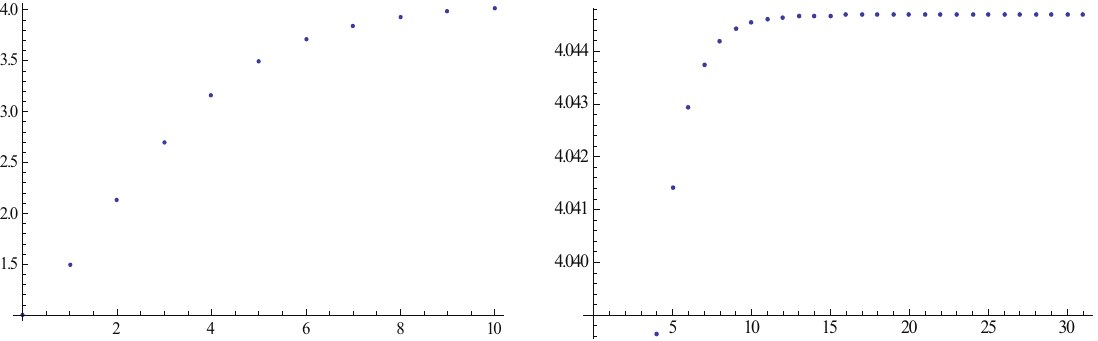}}
    \caption{\captionsize A Mathematica plot of the defect vector $u$ of Example~\ref{exm:unbounded-B-matrix-defect} and Lemma~\ref{thm:defect-vector-has-finite-energy}. The left plot shows $u(x)$ for $x=0,1,\dots,10$, and the plot on the right shows data points for $u(x)$, $x=10,11,12,\dots$.}
    \label{fig:defect-vector}
  \end{figure}


\section{Finite approximants}
\label{sec:finite-approximants}

As mentioned in the previous section, when considering \Lap as an operator on a reproducing kernel Hilbert space \sH, it is not possible to use the matrix representation of \eqref{eqn:Lap-as-matrix-multiplication} because $\{\gd_x\}_{x \in X}$ is not an onb for \sH. Therefore, we change to a different representation of \sH as laid out in \S\ref{sec:repkernel-energy-space}. 

In this section, we return to the setting of \S\ref{sec:repkernel-energy-space}, where $X$ is any (infinite) set, \sQ is a quadratic form on functions on $X$, and $\sH = \dom\sQ/\ker\sQ$ is a Hilbert space with (relative) reproducing kernel $\{v_x\}_{x \in X}$.
For studying infinite sets $X$, it will be helpful to consider a filtration by finite subsets, partially ordered by inclusion. With this aim, we pick a finite subset $F \ci X$ and study the ``restriction'' of $M$ and functions $u$ to this subset. Note that we do not restrict the \emph{support} of the functions under consideration: we restrict the \emph{index set} of the representing functions $\{v_x\}_{x \in X}$, in the spirit of Karhunen-Lo\`eve; see \cite{Ash}. This is akin to using cutoff functions as Fourier multipliers, and leads to a form of spectral reciprocity between the associated Laplace operator, and its ``inverse'' $M$, in the sense described in \S\ref{sec:spectral-reciprocity}. The exact relationship between $M$ (actually, its diagonalization $D$) and \Lap is made precise in Corollary~\ref{thm:Green-is-inverse-of-Lap-strike-o}; see also Remark~\ref{rem:condensation-of-spec-meas}. 
The application we have in mind is a resistance network as discussed in \S\ref{sec:Matrices-and-self-adjoint-operators-in-the-energy-space} but all results are phrased in the context of the reproducing kernel Hilbert space of \S\ref{sec:repkernel-energy-space} so as to keep the scope of discussion more general. 

\begin{defn}\label{def:V}
  Let $\sV := \spn\{v_x\}_{x \in X \less\{o\}}$ and $\sV(F) := \spn\{v_x \suth x \in F\}$. We also write $\ell^2(F)$ for the subspace of functions in $\ell^2(X)$ whose support is contained in $F$. This may seem trivial when $F$ is finite, but the notation helps distinguish between the two different inner products in use.
\end{defn}

\begin{defn}\label{def:Phi}
  Define $\gF : \ell^2(X) \to \sH$ on $\dom\gF = \spn\{\gd_x\}_{x \in X}$ by $\gF\gd_x = v_x$. 
\end{defn}

\begin{remark}\label{rem:closable}
  The operator \gF is typically not closable. To see this, we show why the adjoint is not typically densely defined. First, pick $\gx \in \spn\{\gd_x\}$ and $u \in \sV$, and compute $\gF^\ad$:
  \linenopax
  \begin{align*}
    \la \gx, \gF^\ad u\ra_{\ell^2}
     = \la \gF \gx, u\ra_{\sH} 
     = \sum_{x \in X} \cj{\gx_x} \la v_x, u\ra_{\sH}
    &= \sum_{x \in X} \cj{\gx_x} (u(x) - u(o)).
  \end{align*}
  So for an equivalence class $u \in \sH$, note that $\gF^\ad u$ is the representative of $u$ that vanishes at $o$. For $u \in \sH$, let us denote by $u^{(0)}$ the representative of $u$ specified by $u(o)=0$, so that $\gF^\ad u = u^{(0)}$. Then 
  \linenopax
  \begin{align*}
    \dom \gF^\ad = \{u \in \sH \suth u^{(0)} \in \ell^2(X)\}.
  \end{align*}
  It is easy to see that this class is not dense in \sH; see \cite{DGG} for examples in the case $\sH = \HE$.
\end{remark}

\begin{defn}\label{def:M}
  For a finite set $F \ci X \less \{o\}$, we have $\gF\gx = \sum_{x \in X} \gx(x) v_x$, for all $\gx \in \sV(F)$. Define $M$ to be the matrix of $\gF^\ad\gF$, that is,
  \linenopax
  \begin{equation}\label{eqn:def:M}
    M_{xy} = \ipH[\ell^2]{\gd_x,\gF^\ad\gF\gd_y} = \ipH{v_x,v_y}, \q \forall x,y \in X,
  \end{equation}
  and let $M_F := M\evald{F \times F}$ be the submatrix of $M$ defined by deleting all rows and columns corresponding to points $x \notin F$, i.e., $M_F$ is an $|F| \times |F|$ matrix with entries
  \linenopax
  \begin{equation}\label{eqn:def:MF}
    (M_F)_{xy} = \ipH{v_x,v_y}, \q \forall x,y \in F.
  \end{equation}
  In general, one may have $v_x \in \sV(F)$ with support extending outside of $F$; examples are given in \cite{OTERN}. 
\end{defn}

Note that since $\gx \in \ell^2(F)$ is finitely supported, 
  \linenopax
  \begin{align}\label{eqn:M-prod-as-Phi}
    M \gx(x) 
    = \sum_{y \in F} M_{x,y} \gx(y)
    = \sum_{y \in F} \ipH{v_x,v_y} \gx(y)
    = \ipH{v_x,\sum_{y \in F} \gx(y) v_y}
    = \ipH{v_x, \gF\gx}
    = \gF\gx(x) - \gF\gx(o)
  \end{align}

\begin{defn}\label{def:spec(MF)}
  Denote the spectrum of $M_F$ by $\gL_F = \{\gl_j^F\}$ for some enumeration $j=1,2,\dots,|F|$. Note that $\gL^F > 0$ by Lemma~\ref{thm:M(x,y)-is-pd} and that $M$ is diagonalizable with eigenfunctions $\gx_j = \gx_j^F \in \ell(F) = \ell^2(F)$. That is, the spectral theorem provides an orthonormal basis (onb) $\{\gx_j^F\}$ with
  \linenopax
  \begin{align}\label{eqn:eifs-of-MF}
    M_F \gx_j = \gl_j \gx_j \q \text{for each } j,F.
  \end{align}
  For convenience, we often suppress the index and write \eqref{eqn:eifs-of-MF} as $M_F \gx_\gl = \gl \gx_\gl$.
\end{defn}

\begin{defn}\label{def:uj-onb}
  For a finite $F \ci X$, and $M_F \gx_\gl = \gl \gx_\gl$ as above, define
  \linenopax
  \begin{equation}\label{eqn:def:uj-onb}
    u_\gl := \frac1{\sqrt{\gl}} \sum_{x \in F} \gx_\gl(x) v_x.
  \end{equation}
\end{defn}

\begin{lemma}\label{thm:uj-is-an-onb-for-H(F)}
  The operator $\gY_F:\ell^2(F) \to \sV(F)$ defined by $\gY_F(\gx_\gl) = u_\gl$ is unitary, and consequently $\{u_\gl\}_{\gl \in \gL_F}$ is an orthonormal basis in $\sV(F)$.
  \begin{proof}
    For $x,y \in F$, compute
    \linenopax
    \begin{align*}
      \ipH{u_j, u_k}
      &= \frac1{\sqrt{\gl_j \gl_k}} \sum_{x,y \in F} \cj{\gx_j(x)} \, \gx_k(y) \ipH{v_x,v_y} 
      = \frac1{\sqrt{\gl_j \gl_k}} \sum_{x \in F} \cj{\gx_j(x)} (M_F \gx_k)(x), 
    \end{align*}
    and since $\gx_k$ is an eigenvector, this continues as
    \linenopax
    \begin{align*}
      \ipH{u_j, u_k}
      &= \frac{\sqrt{\gl_k}}{\sqrt{\gl_j}} \sum_{x \in F} \cj{\gx_j(x)} \, \gx_k(x)
      = \sqrt{\frac{\gl_k}{\gl_j}} \ipH[\ell^2]{\gx_j,\gx_k}
      = \gd_{jk},
    \end{align*}
    where $\gd_{jk}$ is the Kronecker delta, since $\{\gx_\gl\}$ is an onb for $\ell^2(F)$.
  \end{proof}
\end{lemma}

\begin{defn}\label{def:P_F}
  By Lemma~\ref{thm:uj-is-an-onb-for-H(F)}, we may let $P_F$ be the projection to $\spn\{u_\gl\}_{\gl \in \gL_F}$. In Dirac notation, this is
  \begin{equation}\label{eqn:def:P_F}
    P_F = \sum_{\gl \in \gL_F} |u_\gl\ra \la u_\gl |.
  \end{equation}
  Note that $P_F$ is projection to $\sV(F)$.
\end{defn}

\begin{lemma}\label{thm:vx-in-terms-of-u-onb}
  With respect to the onb $\{u_\gl\}$, one has
  \linenopax
  \begin{equation}\label{eqn:vx-in-terms-of-u-onb}
    P_F v_x = \sum_{\gl \in \gL_F} \gl^{1/2} \cj{\gx_\gl(x)} u_\gl, \qq \text{for all } x \in F.
  \end{equation}
  \begin{proof}
    Let $x \in F$. Then compute
    \linenopax
    \begin{align*}
      P_F v_x
      = \sum_{\gl \in \gL_F} |u_\gl\ra \la u_\gl | v_x \ra
      = \sum_{\gl \in \gL_F} \ipH{u_\gl,v_x} u_\gl
      = \sum_{\gl \in \gL_F} \frac1{\sqrt{\gl}} \sum_{y \in F} \gx_\gl(y) \ipH{v_y,v_x} u_\gl
    \end{align*}
    by \eqref{eqn:def:P_F} followed by \eqref{eqn:def:uj-onb}. Continuing,
    \linenopax
    \begin{align*}
      P_F v_x 
      = \sum_{\gl \in \gL_F} \frac1{\sqrt{\gl}} \cj{(M_F \gx_\gl)(x)} u_\gl 
      = \sum_{\gl \in \gL_F} \frac{\cj{\gl}}{\sqrt{\gl}} \cj{\gx_\gl(x)} u_\gl,
    \end{align*}
    since $\gx_\gl$ is an eigenvector. Note that $\gl \in \bR^+$, since $M$ is positive semidefinite by assumption. It remains to observe that $P_F v_x = v_x$ for $x \in F$, but this follows from Definition~\ref{def:P_F}.
  \end{proof}
\end{lemma}

\begin{remark}\label{rem:observation-about-v}
  In the language of Theorem~\ref{thm:vNeu-construction}, equation \eqref{eqn:vx-in-terms-of-u-onb} takes the following form:
  \linenopax
  \begin{equation}\label{eqn:v-as-tensor}
    v = \sum_{\gl \in \gL_F} \sqrt{\gl} \left(\cj{\gx_\gl} \otimes u_\gl\right)
  \end{equation}
  where $\{\gx_\gl\}$ is an onb for $\ell^2(F)$ and $\{u_\gl\}$ is an onb for $\sV(F)$. The significance of this symmetric expression of $v$ is that it allows us to compute a norm in \sH (where the sum would be over $x \in F$) by instead computing an $\ell^2$ norm (where the sum is over $\gl \in \gL_F$). For an example, see Corollary~\ref{thm:bounds-on-eigenprojection}.
  
  In \cite{bdG}, the authors show that for $\sH = \HE$ one can construct a Gel'fand triple $\Schw \ci \HE \ci \Schw'$, isometrically embed $\HE \hookrightarrow L^2(\Schw',\prob)$. Here \Schw is a space of ``test functions'' which is dense in \HE, but comes equipped with a strictly finer Fr\'echet topology, and $\Schw'$ is a space of ``distributions'' obtained by taking the dual with respect to this topology. Elements $u \in \HE$ can then be extended to functions on $\Schw'$ via $\tilde u(\gx) = \la u, \gx\ra_\energy$ for $\gx \in \Schw'$. As \prob is a probability measure, one can then interpret $\{v_x\}_{x \in G}$ as a stochastic process, i.e., a system of random variables indexed by the vertices of the underlying graph. In this context, \eqref{eqn:v-as-tensor} becomes an instance of the Karhunen-Lo\`eve decomposition (see, e.g. \cite{Ash}) of a stochastic process into its random and deterministic components:
  \linenopax
  \begin{equation}\label{eqn:v-as-Karhunen}
    \tilde v_x(\gx) = \sum_{\gl \in \gL_F} \sqrt{\gl} \left(\cj{\gx_\gl}(x) \otimes \tilde u_\gl (\gx)\right)
    \qq x \in G, \gx \in (\Schw',\prob).
  \end{equation}
  In fact, it turns out that $\{\tilde u_\gl\}_{\gl \in \gL_F}$ is a system of independent identically distributed Gaussian $N(0,1)$ random variables, for any finite $F \ci X$. See also \S\ref{sec:spectral-measures} for more relations to Karhunen-Lo\`eve.
\end{remark}

\subsection{Spectral reciprocity}
\label{sec:spectral-reciprocity}

In this section, we explore the relationship between $M$ and \Lap. 
In particular, the Spectral Reciprocity Theorem (Theorem~\ref{thm:spectral-reciprocity}) shows how $M$ and \Lap are (almost) inverse operators, and explains why the eigenvalues of $M$ are (almost) the reciprocals of the eigenvalues of \Lap.
%
%
  %

\begin{defn}\label{def:diagn(MF)}
  Denote the diagonalization of $M_F$ by
  \linenopax
  \begin{equation}\label{eqn:def:D-diag(M)}
    D_F := \bigoplus_{\gl \in \gL_F} \gl P_{u_\gl} =
    \left[\begin{array}{rrrr}
      \gl_1 \\ & \gl_2 \\ && \ddots \\ &&& \gl_{|F|}
    \end{array}\right],
  \end{equation}
  where 
  $P_{u_\gl}$ is projection to $\spn\{u_\gl\}_{\gl \in \gL_F}$. Note that $D_F^{-1}$ is a well-defined operator on $\ell^2(F)$ of rank $|F|<\iy$.
\end{defn}

\begin{defn}\label{def:origin-projection}
  Let $P_F^o:\sH \to \sH$ be the projection of $\gd_o$ to $\sV(F)$. That is, $P_F^o = |P_F \gd_o\ra \la P_F \gd_o|$ in Dirac notation.
\end{defn}

\begin{defn}\label{def:Expectation-of-eif}
  If $\{\gx_\gl\}$ is the onb of eigenvectors of $M_F$, denote the expectation of $\gx_\gl$ by
  \linenopax
  \begin{equation}\label{eqn:def:Expectation-of-eif}
    \Ex(\gx_\gl) = \sum_{x \in F} \gx_\gl(x) = \ipH[\ell^2]{\charfn{F}, \gx_\gl}.
  \end{equation}
\end{defn}

\begin{lemma}\label{thm:projn-to-origin}
  If $\gd_o$ is a Dirac mass at the origin, the expansion of $P_F \gd_o$ with respect to $\{u_\gl\}$ is given by
  \linenopax
  \begin{equation}\label{eqn:PFo-in-ul}
    P_F \gd_o = - \sum_{\gl \in \gL_F} \frac{\cj{\Ex(\gx_\gl)}}{\sqrt{\gl}} u_\gl.
  \end{equation}
  \begin{proof}
    Using $P_F = P_F^\ast$, $\eqref{eqn:def:uj-onb}$, and the fact that $u_\gl \in \sV(F)$, we compute the coefficients:
    \linenopax
    \begin{align*}
      \ipH{u_\gl,P_F \gd_o}
      &= \ipH{P_F u_\gl,\gd_o} 
      = \ipH{u_\gl,\gd_o}      
      = \frac1{\sqrt{\gl}} \sum_{x \in F} \cj{\gx_\gl(x)}  \ipH{v_x,\gd_o}
      = -\frac1{\sqrt{\gl}} \sum_{x \in F} \cj{\gx_\gl(x)}.
    \end{align*}
    where the last line follows by Lemma~\ref{thm:vx-is-a-dipole}, since $x \neq o$.
  \end{proof}
\end{lemma}

\begin{defn}\label{def:compression-of-Lap}
  The \emph{compression} of \Lap to $F$ is the restricted action of the operator \Lap to $\sV(F)$, and it is given by $P_F \Lap P_F$.
\end{defn}

\begin{theorem}[Spectral reciprocity]\label{thm:spectral-reciprocity}
  If $F \ci X \less\{o\}$ is nonempty and finite, then 
  \linenopax
  \begin{align}\label{eqn:spectral-reciprocity}
    P_F \Lap P_F = \gF D_F^{-1} \gF^\ad + P_F^o.
  \end{align}
  \begin{proof}
    For $\gl,\gk \in \gL_F$, we have $\ipH{u_\gl, P_F \Lap P_F u_\gk} = \ipH{u_\gl,\Lap u_\gk}$ because $u_\gk \in \sV(F)$. Then 
    \linenopax
    \begin{align}
      \label{eqn:spectral-reciprocity-1}
      \ipH{u_\gl, P_F \Lap P_F u_\gk}
      &= \frac1{\sqrt{\gl\gk}} \sum_{x,y \in F} \cj{\gx_\gl(x)} \gx_\gk(y) \ipH{v_x,\Lap v_y}
      = \frac1{\sqrt{\gl\gk}} \sum_{x,y \in F} \cj{\gx_\gl(x)} \gx_\gk(y) (\gd_{xy} + 1)
    \end{align}
    by \eqref{eqn:def:uj-onb} and \eqref{eqn:<vx,Lap(vy)>=gd(xy)+1}. The computation of \eqref{eqn:spectral-reciprocity-1} continues as 
    \linenopax
    \begin{align}\label{eqn:spectral-reciprocity-2}
      = \frac1{\sqrt{\gl\gk}} \sum_{x \in F} \cj{\gx_\gl(x)} \gx_\gk(x)
        + \frac1{\sqrt{\gl}} \sum_{x \in F} \cj{\gx_\gl(x)} 
          \frac1{\sqrt{\gk}} \sum_{y \in F} \gx_\gk(y) 
      = \frac1{\sqrt{\gl\gk}} \sum_{x \in F} \cj{\gx_\gl(x)} \gx_\gk(x)
        + \frac1{\sqrt{\gl}} \cj{\Ex(\gx_\gl)} 
          \frac1{\sqrt{\gk}} \Ex(\gx_\gk).
    \end{align}
    Since $u_\gl$ is in $\dom \gF^\ad$ automatically for finite $F$, the right side of \eqref{eqn:spectral-reciprocity} is
    \linenopax
    \begin{align*}
      \ipH{u_\gl, (\gF D_F^{-1} \gF^\ad + P_F^o) u_\gk}
      &= \ipH[\ell^2(F)]{\gF^\ad u_\gl, D_F^{-1} \gF^\ad u_\gk} + \ipH{u_\gl, P_F^o u_\gk},
    \end{align*}
    which matches with the right side of \eqref{eqn:spectral-reciprocity-2}, 
    by \eqref{eqn:PFo-in-ul}. This verifies \eqref{eqn:spectral-reciprocity} on the onb of Lemma~\ref{thm:uj-is-an-onb-for-H(F)}, and hence for all of $\sV(F)$.
  \end{proof}
\end{theorem}

\begin{remark}\label{rem:spectral-reciprocity}
  We refer to Theorem~\ref{thm:spectral-reciprocity} as the Spectral Reciprocity Theorem because it relates the eigenvalues of \Lap to the reciprocal eigenvalues of its inverse, on any finite $F \ci X$. 

  Suppose one writes the matrix for \Lap as in Definition~\ref{def:matrix-Laplacian}, so that rows and columns are indexed by points of $X$. Let $\tilde\Lap$ be the matrix which results from deleting the row and column corresponding to a chosen point $o$. 
  Corollary~\ref{thm:Green-is-inverse-of-Lap-strike-o} makes precise the well-known statement that one can invert the Laplacian after deleting the row and column corresponding to a point $o$.%
  \footnote{Recall that if $M$ is a Hermitian matrix acting on a finite-dimensional Hilbert space $H$, then the restriction of $M$ to the orthocomplement of the zero eigenspace is invertible.} 
  In particular, without deleting the row and column of $o$, one is forced to contend with an auxiliary term 1 in \eqref{eqn:<vx,Lap(vy)>=gd(xy)+1} (which corresponds to the projection $P^o = |o\ra \la o|$ to the 1-dimensional subspace spanned by $\gd_0$).
\end{remark}

\begin{lemma}\label{thm:lim(LapFn)=Lap-P1}
  For every nested sequence of finite sets $\{F_n\}_{n \in \bN}$ with $\bigcup F_n = X \less \{o\}$, the limit of $P_{F_n} \Lap P_{F_n}$ exists and with $\dom \Lap$ as in \eqref{eqn:Laplacian-domain},
  \linenopax
  \begin{equation}\label{eqn:lim(LapFn)=Lap-P1}
    \Lap = \lim_{n \to \iy} P_{F_n} \Lap P_{F_n},
  \end{equation}
  in the strong operator topology, that is, $\lim_{n \to \iy} \|P_{F_n} \Lap P_{F_n} v - \Lap v\|_\sH$ for all $v \in \dom \Lap$. 
  \begin{proof}
    Let $f \in \dom \Lap$ so that there is some finite set $F \ci X \less \{o\}$ for which
    \linenopax
    \begin{align*}
      f = \sum_{x \in F} \gx_x v_x.
    \end{align*}
    Without loss of generality, let $\{F_n\}_{n=1}^\iy$ be an exhaustion of $X \less \{o\}$ with $F \ci F_1$. Then $P_{F} f = f$, and
    \linenopax
    \begin{align*}
      P_{F_n} \Lap P_{F_n} f = P_{F_n} \Lap f \limas{n} \Lap f,
    \end{align*}
    since $P_{F_n}$ increases to the identity operator:
    \linenopax
    \begin{align*}
      \|P_{F_n} \Lap P_{F_n} f - \Lap f\|_\sH
      &= \|(P_{F_n} -\id) \Lap f\|_\sH 
       \limas{n \to \iy} 0.
      \qedhere
    \end{align*}
  \end{proof}
\end{lemma} 

\begin{defn}\label{def:P^o}
  Let $P^o:=|\gd_o\ra \la\gd_o|=\proj{\gd_o}$ be the rank-1 projection on \sH defined  by $\ipH{u,P^ow} = \ipH{u,\gd_o}\ipH{\gd_o,w}$. 
\end{defn}

\begin{cor}\label{thm:Green-is-inverse-of-Lap-strike-o}
  The limit $\Lap - P^o = \lim_{n \to \iy} \gF D^{-1}_{F_n}\gF^\ad$ exists, for any exhaustion $\{F_n\}$ of $X \less \{o\}$.
  \begin{proof}
    Since arguments exactly analogous to those in Lemma~\ref{thm:lim(LapFn)=Lap-P1} give $P_{F_n}^o \limas{n} P^o$,
    we have 
    \linenopax
    \begin{align*}
      \lim_{n \to \iy} \gF D_{F_n}^{-1} \gF^\ad
      &= \lim_{n \to \iy} P_{F_n} \Lap P_{F_n} + \lim_{n \to \iy} P_{F_n}^o 
       = \Lap - P^o, 
    \end{align*}
    by applying Theorem~\ref{thm:spectral-reciprocity} and then Lemma~\ref{thm:lim(LapFn)=Lap-P1}.
  \end{proof}
\end{cor}

\subsection{Spectral measures}
\label{sec:spectral-measures}

  
Recall from Definition~\ref{def:Expectation-of-eif} that
  $\Ex(\gx_j) = \sum_{x \in F} \gx_j(x) = \ipH[\ell^2]{\charfn{F}, \gx_j}$,
and that from Lemma~\ref{thm:projn-to-origin}, the expansion of $P_F \gd_o$ with respect to $\{u_\gl\}$ is given by
\linenopax
\begin{equation}\label{eqn:PFo-in-ul-recalled}
  P_F \gd_o = - \sum_{\gl \in \gL_F} \frac{\cj{\Ex(\gx_\gl)}}{\sqrt{\gl}} u_\gl.
\end{equation}

\begin{defn}\label{def:spectrum-of-PFLapPF}
  Since $P_F \Lap P_F = D_F^{-1} + P_F \gd_o$ is the $J \times J$ matrix $T_F$ whose \nth[(j,k)] entry is given by 
  \linenopax
  \begin{align*}
    \gt_{j,k} = \frac{\gd_{j,k}}{\gl_j} + \frac{\cj{\Ex(\gx_j)}\Ex(\gx_j)}{\sqrt{\gl_j \gl_k}},
  \end{align*}
  denote the spectrum of this matrix $T_F = [\gt_{j,k}]$ by $S^F = \spec(T_F)= \{\gs_j^F\}_{j=1}^J$.
\end{defn}

\begin{remark}\label{rem:suppress-the-F}
  In Definition~\ref{def:spectrum-of-PFLapPF}, it is important to note that $\gt_{j,k}, \gx_j$, and $\gl_j$ all depend on the choice of $F$. However, for ease of notation we suppress this dependence and also henceforth write $\gs_j = \gs_j^F$.
\end{remark}

Recall from Definition~\ref{def:spec(MF)} that $\gL_F = \spec(M_F) = \{\gl_J\}_{j=1}^J$.

\begin{cor}\label{thm:eigenexpectations-average-to-1}
  For any finite subset $F \ci X\less\{o\}$ with $|F|=J$, one has
  $  \frac1J \sum_{j=1}^J \Ex(\gx_j)^2 = 1.$
  \begin{proof}
    Since $\charfn{F}$ coincides with the constant vector \one on $F$, we use $P_{\gx_j} u = \ipH[\ell^2]{\gx_j,u}\gx_j$ to compute directly
    \linenopax
    \begin{align*}
      \sum_j \left|\Ex(\gx_j) \right|^2
      &= \sum_j \left|\ipH[\ell^2]{\charfn{F}, \gx_j}\right|^2
      = \sum_j \|P_{\gx_j} \charfn{F}\|^2
      = \|\charfn{F}\|_2^2
      = |F|
      = J.
      \qedhere
    \end{align*}
  \end{proof}
\end{cor}

\begin{cor}\label{thm:bounds-on-eigenprojection}
  For any finite subset $F \ci X\less\{o\}$ with $|F|=J$, one has
  $  \frac{J}{\max \gl} \leq \|P^o_F\| \leq \frac{J}{\min \gl}.$ 
  \begin{proof}
    Using Corollary~\ref{thm:eigenexpectations-average-to-1} and Definition~\ref{def:origin-projection}, 
    $  \|P^o_F\| 
      = \|P_F \gd_o\|_{\sH}^2 
      = \sum_{j=1}^J \Ex(\gx_j)^2 
      = J.$
    See Remark~\ref{rem:observation-about-v}. Then the double inequality follows by estimating by the largest (but clearly finite) eigenvalue and the smallest (but clearly strictly positive) eigenvalue.
  \end{proof}
\end{cor}

\begin{remark}\label{rem:condensation-of-spec-meas}
  When \Lap is not essentially self-adjoint, the presence of $P_F^o$ (as in \eqref{eqn:spectral-reciprocity}, for example) makes it impossible to obtain self-adjoint extensions of \Lap via a filtration by finite subsets. This obstacle can only be overcome by passing to spectral measures. 
  If $\dom \Lap$ is as in \eqref{eqn:Laplacian-domain}, then the spectral measure of some self-adjoint extension of \Lap comes from the weak-$\ad$ limit of linear combinations of of equally weighted Dirac masses:
\linenopax
\begin{equation}\label{eqn:spec-premeas-on-F}
  \gm_F = \frac1J \sum_{j=1}^J \gd_{\gs_j}.
\end{equation}

Here, $\gm_F$ refers to the spectral representation of $P_F \Lap P_F$, and we are relying on standard tools from the literature. Indeed, approximation of measures with the use of spectral sampling is a versatile and powerful tool. For approximation in the weak-$\ad$ topology on measures (as in the present context), see the excellent reference book \cite{Billingsley} for details. When applied to spectral measures, these approximations were first studied in the book by M. Stone; see \cite[Ch.~X]{Stone}. The approach in \cite{Stone} is especially amenable to our present applications: a main theme is the study of unbounded operators in Hilbert space, realized concretely as banded infinite matrices.
This is illustrated in the following diagram: 
\linenopax
\begin{align}\label{eqn:spec-measures-as-limits}
  \xymatrix{
    P_{F} \Lap P_{F} \ar[rr]^{F \to X} && \Lap \ar[r]^{\ci} & \tilde\Lap \\
    \gm_F \ar[rrr]^{\textrm{weak-*}} \ar@{<->}[u] &&& \tilde \gm \ar@{<->}[u]
  }
\end{align}
In the limit of \eqref{eqn:spec-premeas-on-F} as $F \to X$, 
may $\gm_F$ become a smooth measure. 
The key point is that considering the limit of $P_F \Lap P_F$ as $F \to X$ does not take one far enough. However, consideration of the spectral measures $\gm_F$ of $P_F \Lap P_F$ shows that \emph{each} weak-* limit $\tilde \gm$ is the spectral measure of some self-adjoint extension $\tilde \Lap$ of \Lap, and by general theory, every self-adjoint extension of \Lap arises in this way. 

In the preceding discussion, $F \to X$ refers implicitly to a limit with respect to an exhaustion $\{F_n\}_{n \in \bN}$, where $F_k \ci F_{k+1}$ and $\bigcup_{n=1}^\iy F_k = X \less \{o\}$. Note that the limit $\Lap = \lim_{F \to X} P_{F} \Lap P_{F}$ is unique (see Lemma~\ref{thm:lim(LapFn)=Lap-P1}) and hence independent of the choice of exhaustion $\{F_n\}_{n \in \bN}$. 
However, the nonuniqueness of weak-* limits corresponds to the fact that $\tilde \gm = \lim_{F \to X} \gm_F$ may depend on the choice of exhaustion. Different weak-* limits may correspond to different self-adjoint extensions $\tilde \Lap$ of \Lap. 
\end{remark}

\subsection{Spectral reciprocity for balanced functions}
\label{sec:balanced}

Balanced functions are functions which sum to $0$. In the context of resistance networks (see \S\ref{sec:Matrices-and-self-adjoint-operators-in-the-energy-space}), a balanced function is the divergence of a current flow with no transient component; these functions are mentioned briefly in \cite[\S{III}.3]{Soardi94}.
 
\begin{defn}\label{def:B}
  A function $\gx:X \to \bC$ is \emph{balanced} iff \gx has finite support and $\sum_{x \in X} \gx(x) = 0$. Denote the space of such functions by \sB. For any finite $F \ci X\less\{o\}$, let $\sB_F$ denote the collection of functions in \sB whose support is contained in $F$. 
\end{defn}

Recall from Definition~\ref{def:V} that $\sV := \spn\{v_x\}_{x \in X \less\{o\}}$ and $\sV(F) := \spn\{v_x \suth x \in F\}$, and from Definition~\ref{def:Phi} that $\gF : \ell^2(X) \to \sH$ is given by $\gF(\gd_x) = v_x$ on $\dom\gF = \spn\{\gd_x\}_{x \in X}$.

\begin{defn}\label{def:V0}
  Denote the subspace of \sV with balanced coefficients by 
  \linenopax
  \begin{align}\label{eqn:V0}
    \sV_0 := \gF(\sB) = \{\gF(\gx) \suth \gx \in \sB\},
  \end{align} 
  and similarly for $\sV_0(F) := \gF(\sB_F)= \{\gF(\gx) \suth \gx \in \sB_F\}$.
\end{defn}

The following curious fact can be found in most introductory books on functional analysis.

\begin{prop}\label{thm:discontinuous}
  Let $A$ be a topological vector space, and let $A_0$ be a dense linear subspace. If $f$ is a linear functional on $A_0$, then $\ker f$ is dense in $A$ if and only if $f$ is discontinuous.
\end{prop}

\begin{lemma}\label{thm:zero-sum-density}
  $\sB$ is dense in $\ell^2(X)$ if and only if $X$ is infinite.
  \begin{proof}
    Define $f:\sB \to \bC$ by $f(\gx) = \sum_{x \in X} \gx_x$. Note that $X$ is finite if and only if the constant function \one is in $\ell^2(X)$, which (by Riesz duality) holds if and only if $f$ is continuous on $\ell^2(X)$. 
    The result now follows from Proposition~\ref{thm:discontinuous}.
  \end{proof}
\end{lemma}

The next lemma indicates how \gF ``intertwines'' the spectral densities of \Lap and $M$.

\begin{lemma}\label{thm:Phi-intertwines}
  For all $\gx \in \sB$, one has $\la \gF(\gx), \Lap \gF(\gx) \ra_\sH = \|\gx\|_{\ell^2}^2$ and $\la\gx, M\gx \ra_{\ell^2} = \| \gF(\gx)\|_\sH^2$, and hence
  \linenopax
  \begin{align}\label{eqn:Phi-intertwines}
    \frac{\la \gF(\gx), \Lap \gF(\gx) \ra_\sH}{\| \gF(\gx)\|_\sH^2} 
    = \frac{\|\gx\|_{\ell^2}^2}{\la\gx, M\gx \ra_{\ell^2}}.
  \end{align}
  \begin{proof}  
    The first identity is immediate for $\gx \in \sB$ by \eqref{eqn:<u,Lapu>=2sums}. For the second, note that 
    \linenopax
    \begin{align*}
      \| \gF(\gx)\|_\sH^2
      & = \left\la \sum_{x \in X} \gx(x) v_x, \sum_{y \in X} \gx(y) v_y \right\ra_\sH
      = \sum_{x \in X} \sum_{y \in X} \cj{\gx(x)} \gx(y) M_{x,y}
      = \la \gx, M \gx\ra_{ell^2}.
      \qedhere
    \end{align*}
  \end{proof}
\end{lemma}

\begin{defn}\label{def:gap}
  We say that \Lap has a spectral gap $\ga > 0$ iff
  \linenopax
  \begin{align}\label{eqn:gap}
      \ga \| \gy\|_\sH^2 \leq \la \gy, \Lap\gy\ra_\sH,
      \qq\text{for all } \gy \in \sV_0.
  \end{align}
\end{defn}

\begin{theorem}[Spectral gap]
  \label{thm:spectral-gap-balanced}
  \Lap has a spectral gap $\ga > 0$ if and only if \eqref{eqn:def:M} defines a bounded self-adjoint operator $M: \ell^2(X) \to \ell^2(X)$ with $\|M\| \leq \frac1{\sqrt{\ga}}$.
  \begin{proof}  
    This follows immediately when either side of \eqref{eqn:Phi-intertwines} is bounded from below by $\ga > 0$.
  \end{proof}
\end{theorem}

\subsection*{Acknowledgements}
The authors are grateful to Ilwoo Cho, Raul Curto, Dorin Dutkay, Matthias Keller, Paul Muhly, Myung-Sin Song, and Rados{\l}aw Wojciechowski for helpful conversations, suggestions, and recommendations on the literature. We are also grateful to the referee for a careful review and detailed comments.

\bibliographystyle{alpha}
\bibliography{SRAMO}

\begin{thebibliography}{MYY94}

\bibitem[Aro50]{Aronszajn50}
N.~Aronszajn.
\newblock Theory of reproducing kernels.
\newblock {\em Trans. Amer. Math. Soc.}, 68:337--404, 1950.

\bibitem[Ash65]{Ash}
Robert Ash.
\newblock {\em Information theory}.
\newblock Interscience Tracts in Pure and Applied Mathematics, No. 19.
  Interscience Publishers John Wiley \& Sons, New York-London-Sydney, 1965.

\bibitem[BCR84]{Ber84}
Christian Berg, Jens Peter~Reus Christensen, and Paul Ressel.
\newblock {\em Harmonic analysis on semigroups}, volume 100 of {\em Graduate
  Texts in Mathematics}.
\newblock Springer-Verlag, New York, 1984.
\newblock Theory of positive definite and related functions.

\bibitem[BF06]{BakerFaber}
Matthew Baker and Xander Faber.
\newblock Metrized graphs, {L}aplacian operators, and electrical networks.
\newblock In {\em Quantum graphs and their applications}, volume 415 of {\em
  Contemp. Math.}, pages 15--33. Amer. Math. Soc., Providence, RI, 2006.

\bibitem[Bil99]{Billingsley}
Patrick Billingsley.
\newblock {\em Convergence of probability measures}.
\newblock Wiley Series in Probability and Statistics: Probability and
  Statistics. John Wiley \& Sons Inc., New York, second edition, 1999.
\newblock A Wiley-Interscience Publication.

\bibitem[BR07]{BakerRumely}
Matt Baker and Robert Rumely.
\newblock Harmonic analysis on metrized graphs.
\newblock {\em Canad. J. Math.}, 59(2):225--275, 2007.

\bibitem[Chu01]{Chu01}
Fan Chung.
\newblock {\em Spectral Graph Theory}.
\newblock Cambridge, 2001.

\bibitem[DS88]{DuSc88}
Nelson Dunford and Jacob~T. Schwartz.
\newblock {\em Linear operators. {P}art {II}}.
\newblock Wiley Classics Library. John Wiley \& Sons Inc., New York, 1988.

\bibitem[F{\=O}T94]{FOT94}
Masatoshi Fukushima, Y{\=o}ichi {\=O}shima, and Masayoshi Takeda.
\newblock {\em Dirichlet forms and symmetric {M}arkov processes}, volume~19 of
  {\em de Gruyter Studies in Mathematics}.
\newblock Walter de Gruyter \& Co., Berlin, 1994.

\bibitem[Jor08]{Jorgensen08}
Palle E.~T. Jorgensen.
\newblock Essential self-adjointness of the graph-{L}aplacian.
\newblock {\em J. Math. Phys.}, 49(7):073510, 33, 2008.

\bibitem[JP09a]{DGG}
Palle E.~T. Jorgensen and Erin P.~J. Pearse.
\newblock A discrete {G}auss-{G}reen identity for unbounded {L}aplace operators
  and transience of random walks.
\newblock {\em In review}, pages 1--25, 2009.
\newblock \arxiv{0906.1586}.

\bibitem[JP09b]{OTERN}
Palle E.~T. Jorgensen and Erin P.~J. Pearse.
\newblock Operator theory and analysis of infinite resistance networks.
\newblock pages 1--247, 2009.
\newblock \arxiv{0806.3881}.

\bibitem[JP09c]{RBIN}
Palle E.~T. Jorgensen and Erin P.~J. Pearse.
\newblock Resistance boundaries of infinite networks. \textit{{T}o
  appear:$\negsp[7]$}.
\newblock In {\em Boundaries and Spectral Theory}. Birkhauser, 2009.
\newblock 32 pages. \arxiv{0909.1518}.

\bibitem[JP09d]{bdG}
Palle E.~T. Jorgensen and Erin P.~J. Pearse.
\newblock Stochastic integration and boundaries of infinite networks.
\newblock {\em In review}, 2009.
\newblock 31 pages. \arxiv{0906.2745}.

\bibitem[JP10a]{ERM}
Palle E.~T. Jorgensen and Erin P.~J. Pearse.
\newblock A {H}ilbert space approach to effective resistance metrics.
\newblock {\em Complex Anal. Oper. Theory}, 4(4):975--1030, 2010.
\newblock \arxiv{0906.2535}.

\bibitem[JP10b]{Interpolation}
Palle E.~T. Jorgensen and Erin P.~J. Pearse.
\newblock Interpolation on resistance networks.
\newblock 2010.
\newblock 14 pages. In preparation.

\bibitem[JP10c]{Multipliers}
Palle E.~T. Jorgensen and Erin P.~J. Pearse.
\newblock Multiplication operators on the energy space.
\newblock {\em To appear: Journal of Operator Theory}, 2010.
\newblock 25 pages. \arxiv{1007.3516}.

\bibitem[JP10d]{LPS}
Palle E.~T. Jorgensen and Erin P.~J. Pearse.
\newblock Scattering theory on resistance networks.
\newblock 2010.
\newblock 13 pages. In preparation.

\bibitem[Kig01]{Kig01}
Jun Kigami.
\newblock {\em Analysis on fractals}, volume 143 of {\em Cambridge Tracts in
  Mathematics}.
\newblock Cambridge University Press, Cambridge, 2001.

\bibitem[Kig03]{Kig03}
Jun Kigami.
\newblock Harmonic analysis for resistance forms.
\newblock {\em J. Funct. Anal.}, 204(2):399--444, 2003.

\bibitem[Kig09]{Kig09}
Jun Kigami.
\newblock Resistance forms, quasisymmetric maps and heat kernel estimates.
\newblock preprint:96, 2009.

\bibitem[KL09]{KellerLenz09}
Matthias Keller and Daniel Lenz.
\newblock Dirichlet forms and stochastic completeness of graphs and subgraphs.
\newblock {\em Preprint}, 2009.
\newblock \arxiv{0904.2985}.

\bibitem[KL10]{KellerLenz10}
Matthias Keller and Daniel Lenz.
\newblock Unbounded {L}aplacians on graphs: basic spectral properties and the
  heat equation.
\newblock {\em Math. Model. Nat. Phenom.}, 5(4):198--224, 2010.

\bibitem[LP10]{Lyons:ProbOnTrees}
Russell Lyons and Yuval Peres.
\newblock {\em Probability on Trees and Graphs}.
\newblock Unpublished (see Lyons' web site), 2010.

\bibitem[MYY94]{MuYaYo}
Atsushi Murakami, Maretsugu Yamasaki, and Yoshinori {Yone-E}.
\newblock Some properties of reproducing kernels on an infinite network.
\newblock {\em Mem. Fac. Sci. Shimane Univ.}, 28:1--8, 1994.

\bibitem[PS72]{PaSc72}
K.~R. Parthasarathy and K.~Schmidt.
\newblock {\em Positive definite kernels, continuous tensor products, and
  central limit theorems of probability theory}.
\newblock Lecture Notes in Mathematics, Vol. 272. Springer-Verlag, Berlin,
  1972.

\bibitem[RS72]{ReedSimonI}
Michael Reed and Barry Simon.
\newblock {\em Methods of modern mathematical physics. {I}. {F}unctional
  analysis}.
\newblock Academic Press, New York, 1972.

\bibitem[Soa94]{Soardi94}
Paolo~M. Soardi.
\newblock {\em Potential theory on infinite networks}, volume 1590 of {\em
  Lecture Notes in Mathematics}.
\newblock Springer-Verlag, Berlin, 1994.

\bibitem[Sto90]{Stone}
Marshall~Harvey Stone.
\newblock {\em Linear transformations in {H}ilbert space}, volume~15 of {\em
  American Mathematical Society Colloquium Publications}.
\newblock American Mathematical Society, Providence, RI, 1990.
\newblock Reprint of the 1932 original.

\bibitem[vN32]{vN32a}
J.~von Neumann.
\newblock \"{U}ber adjungierte {F}unktionaloperatoren.
\newblock {\em Ann. of Math. (2)}, 33(2):294--310, 1932.

\bibitem[Web09]{Web08}
Andreas Weber.
\newblock Analysis of the laplacian and the heat flow on a locally finite
  graph.
\newblock {\em J. Math. Anal. and Appl.}, 370(1):146--158, 2009.
\newblock \arxiv{0801.0812}.

\bibitem[Woj07]{Woj07}
Rados{\l}aw~K. Wojciechowski.
\newblock Stochastic completeness of graphs.
\newblock {\em Ph. D. Dissertation}, 2007.
\newblock 72 pages. \arxiv{0712.1570}.

\bibitem[Woj09]{Woj09}
Rados{\l}aw~K. Wojciechowski.
\newblock Heat kernel and essential spectrum of infinite graphs.
\newblock {\em Indiana Univ. Math. J.}, 58(3):1419--1441, 2009.
\newblock \arxiv{0802.2745}.

\end{thebibliography}

\end{document}